\documentclass[a4paper,11pt,leqno]{amsart}
\usepackage[lmargin=3cm,tmargin=3.5cm,bmargin=3.5cm,rmargin=3cm]{geometry} 
\usepackage{amsaddr}
\usepackage{graphicx}
\usepackage{subfig}
\usepackage{amssymb}
\usepackage{amsfonts}
\usepackage{mathrsfs} 
\usepackage{amsmath}
\usepackage{epsfig}
\usepackage{float}
\usepackage{setspace}
\usepackage{tikz}
\usepackage{enumerate}
\usepackage[colorlinks=true,urlcolor=blue, citecolor=red,linkcolor=blue,linktocpage,pdfpagelabels, bookmarksnumbered,bookmarksopen]{hyperref}
\usepackage{fancyhdr}
\newtheorem{theorem}{Theorem}[section]
\newtheorem{lemma}{Lemma}[section]

\newtheorem{remark}{Remark}[section]
\newtheorem{proposition}{Proposition}[section]

\numberwithin{equation}{section}
\newcommand{\R}{\mathbb R}

\begin{document}

\title[IBVP for Schr\"odinger-Korteweg-de Vries system on the half-line]{The  Initial-Boundary Value Problem for the Schr\"odinger-Korteweg-de Vries system on the Half-Line}

\author[M\'arcio Cavalcante]{M\'arcio Cavalcante}
\address{\emph{Centro de Modelamiento Matem\'atico \\ 
Universidad de Chile, Beauchef 851.\\
Edificio Norte – Piso 7, Santiago - Chile}.}
\email{mcavalcante@cmm.uchile.cl}

\author[A. J. Corcho]{Ad\'an J. Corcho}
\address{\emph{Instituto de Matem\'atica, Universidade Federal do Rio de Janeiro.\\ 
Centro de Tecnologia - Bloco C. Cidade Universit\'aria.\\
Ilha do Fund\~ao  21941-909.  Rio de Janeiro - RJ, Brazil}.}
\email{adan@im.ufrj.br}

\thanks{The first author was partially supported by CAPES, Brazil and Centro de Modelamiento Matem\'atico, Chile.}

\thanks{The second author was partially supported by CAPES and CNPq-309752/2013-2, Brazil}



\date{\today}

\maketitle

\medskip 
\begin{center}
\begin{minipage}{13cm}
{\small \textbf{Abstract.} We prove local well-posedness for the initial-boundary value problem (IBVP) associated to  the Schr\"odinger-Korteweg de Vries system  on right and left half-lines. The results are obtained in the low regularity setting by using two analytic families of boundary forcing operators, being one of these family developed by Holmer to study the IBVP associated to the Korteweg-de Vries equation (Communications in Partial Differential Equations, 31 (2006)) and the other family one was recently introduced by Cavalcante (Differential and Integral Equations, 30 (2017)) in the context of nonlinear Schr\"odinger with quadratic nonlinearities.}
\end{minipage}
\end{center}


\medskip 
\section{\textbf{Introduction}}
The Schr\"odinger-Korteweg-de Vries (NLS-KdV) system is given by the coupled equations:
\begin{equation}\label{SKpuro}
\begin{cases}
iu_t+u_{xx}=\alpha uv+\beta |u|^2u& \text{for}\quad (x,t)\in\mathbb{R}\times(0,T),\\
v_t+v_{xxx}+ vv_x=\gamma (|u|^2)_x& \text{for}\quad (x,t)\in\mathbb{R}\times(0,T),\\
u(x,0)=u_0(x),\ v(x,0)=v_0(x)& \text{for}\quad x\in\mathbb{R},
\end{cases}
\end{equation}
where $u=u(x,t)$ is a complex-valued function, $v=v(x,t)$ is a real valued function and $\alpha,\, \beta, \, \gamma$ are real constants.

More precisely, the system consists in a non-linear coupling of the classical cubic nonlinear Schr\"odinger (NLS) and Korteweg-de Vries (KdV) equations. For a didactic introduction to both models we refer the books \cite{Linares-Ponce}  and \cite{Tao}. From the physical point of view the NLS-KdV system governs the interactions between short-wave, $u=u(x, t)$, and long wave, $v=v(x, t)$, in dispersive media. For  example, the interactions described by \eqref{SKpuro} can arises in some phenomena of fluid mechanics as well as plasma physics. We refer the works \cite{fi1}, \cite{fi2}, \cite{fi3} and \cite{fi4} for a better description of these applications.
Recently, the authors of the works \cite{D-Nuyen-S, Liu-Nguyen} studied more deeply the physical derivation of the model and they ensure that NLS-KdV system can not be deduced from the water wave equations. 

The central theme of this paper is the local well-posedness theory for the NLS-KdV system on the half-line. In order to motivate our study in the context of the half-line, we first recall some known results for system \eqref{SKpuro} about well-posedness in the whole line for initial data $(u_0, v_0)$ in classical Sobolev spaces  
$H^s(\R)\times H^k(\R)$, which we summarize as follows:
\begin{itemize}
	\item M. Tsutsumi \cite{T} showed  global well-posedness  when $s=k+1/2$ with $k\in \mathbb{Z}^+$. \smallskip 
	\item B. Guo and C. Miao \cite{GM} derived global well-posedness for resonant interactions ($\beta= 0$), when $s=k$ with $s\in\mathbb{Z}^+$.\smallskip
	\item D. Bekiranov, T. Ogawa and G. Ponce \cite{BOP} showed local well-posedness, by using the Fourier restriction method, when  $k=s-1/2$ with $s \geq 0$. \smallskip
	\item A. J. Corcho and F. Linares \cite{CL}  improved the local results in \cite{BOP} to the class described by $s\geq 0$ and  $k>-3/4$ provided the conditions:
	$s-1 \leq k \leq 2s-1/2$\, for\, $s\leq 1/2$\, and\,  $s-1\leq k <s+1/2$\, for\, $s>1/2$.\smallskip
	\item H. Pecher \cite{Pecher} showed, when $\alpha \gamma >0$, global well-posedness  for $s=k$ in the cases:  
	$3/5<s<1$ when $\beta =0$  and $2/3<s<1$ when $\beta \neq 0$.\smallskip 
	\item  Y. Wu \cite{YW} extended the results in \cite{CL}, showing local well-posedness for indexes $s\geq 0,\ -3/4<k <4s$ and $s-2<k<s+1$ for any $\beta \in \R$  and  for resonant interactions ($\beta=0$) local theory is developed for indexes $-3/4<k<4s$ and  $s-2<k<s+1$. Moreover,  when $\alpha \gamma >0$, the author proved global well-posedness for $k=s$ with 
	$1/2<s<1$ no matter in the resonant case or in the non-resonant case.\smallskip 
	\item H. Wang and S. Cui \cite{WangCui} developed local theory for $k=-3/4$  with $s=0$, for any $\beta \in \R$,  and with $-3/16< s\le 1/4$ for resonant interactions.\smallskip 
	\item Z. Guo and Y. Wang \cite{GW} showed local-well posedness for $k=-3/4$  with $s=0$, for any $\beta \in \R$,  and with $s>-1/16$ for resonant interactions.
\end{itemize}
A natural question is to know what is the region of Sobolev indexes that describes the local theory for the 
Schr\"odinger-Korteweg-de Vries interactions when the model is considered in the half-line. As we shall see, the problem of well-posedness for half-line is technically more difficult than the whole line.
 
\subsection{The model on the half-line} For the initial-boundary value problem (IBVP) as far as
we know the well-posedness theory of the NLS-KdV system on the half-line is unknown. We shall study  two formulations of the IBVP for the NLS-KdV system. Firstly we study the model on the right half-line 
$\mathbb{R}^+=(0,+\infty)$, namely
\begin{equation}\label{SK}
\begin{cases}
iu_t+u_{xx}=\alpha uv+\beta |u|^2u& \text{for}\quad (x,t)\in(0,+\infty)\times(0,T),\\
v_t+v_{xxx}+ vv_x=\gamma (|u|^2)_x& \text{for}\quad (x,t)\in(0,+\infty)\times(0,T),\\
u(x,0)=u_0(x),\ v(x,0)=v_0(x)& \text{for}\quad x\in (0,+\infty),\\
u(0,t)=f(t),\ v(0,t)=g(t)& \text{for}\quad t\in(0,T).
\end{cases}
\end{equation}

The second formulation studied in this work is the case of the left half-line $\mathbb{R}^-=(-\infty,0)$, given by the following model:
\begin{equation}\label{SKe}
\begin{cases}
iu_t+u_{xx}=\alpha uv+\beta |u|^2u& \text{for}\quad (x,t)\in(-\infty,0)\times(0,T),\\
v_t+v_{xxx}+ vv_x=\gamma (|u|^2)_x& \text{for}\quad (x,t)\in(-\infty,0)\times(0,T),\\
u(x,0)=u_0(x),\ v(x,0)=v_0(x)& \text{for}\quad x\in(-\infty,0),\\
u(0,t)=f(t),\ v(0,t)=g(t),\ v_x(0,t)=h(t)&  \text{for}\quad t\in(0,T).
\end{cases}
\end{equation}

The presence of one boundary condition in the right half-line problem \eqref{SK} versus two boun\-dary conditions in the left half-line problem \eqref{SKe} for the KdV-component of the system can be  motivated by integral identities on smooth solutions to the linear KdV equation $v_t+ v_{xxx}=0$. Indeed, for such $v$ and an arbitrary time $t$ such that $0<t<T$, we have
\begin{equation}\label{unico1}
\int_0^{+\infty}v^2(x,t)dx=\int_0^{+\infty}v^2(x,0)dx + \int_0^t\Big[2v(0,t')v_{xx}(0,t')-v_x^2(0,t')\Big]dt'
\end{equation}
and
\begin{equation}\label{unico2}
\int_{-\infty}^0v^2(x,t)dx=\int_{-\infty}^0v^2(x,0)dx-\int_0^t\Bigl[2v(0,t')v_{xx}(0,t')-v_x^2(0,t')\Bigl]dt'.
\end{equation}

\smallskip
So, assuming $v(x,0)=0=v(0,t')$ for all $x>0$ and $0<t'<t$, we can conclude from \eqref{unico1} that $v(x,t)=0$ for all $x>0$. On the other hand, the possibility of existence of $v(x,t)\neq 0$ for $x<0$  satisfying $v(x,0)=0=v(0,t')$ for all $x<0$ and $0<t'<t$ is not precluded by \eqref{unico2}; indeed, such nonzero solutions do exist as can be seen in Subsection \ref{section4}. However, \eqref{unico2} does show that homogeneous conditions $v(x,0)=v(0,t')=v_x(0,t')=0$ for all $x<0$ and $0<t'<t$ imply that 
$v(x,t)=0$ for all $x<0$. 

\subsection{Functional spaces for the initial-boundary data}
Now we discuss appropriate functional spaces for the initial and boundary data, examining again the behavior of solutions of the linear problem on the line $\mathbb{R}$ as motivation.

On the line $\mathbb{R}$, we define the $L^2$-based inhomogeneous Sobolev spaces $H^s(\mathbb{R})$ equipped with the norm $\|\phi\|_{H^s(\mathbb{R})}=\|\langle\xi\rangle^s\widehat{\phi}(\xi)\|_{L^2(\mathbb{R})}$, where $\langle\xi\rangle=1+|\xi|$ and $\widehat{\phi}$ denotes the Fourier transform of $\phi$. 

The operators $e^{it\partial_x^2}$ and $e^{-t\partial_x^3}$ denote the linear homogeneous solution group associated to the linear Schr\"odinger and KdV equations, respectively, posed on $\mathbb{R}$, i.e.,

\begin{equation}
e^{it\partial_x^2}\phi(x)=\frac{1}{2\pi}\int_{\mathbb{R}}e^{ix\xi}e^{-it\xi^2}\widehat{\phi}(\xi)d\xi
\end{equation}
and
\begin{equation}
e^{-t\partial_x^3}\phi(x)=\frac{1}{2\pi}\int_{\mathbb{R}}e^{i\xi x}e^{it\xi^3}\widehat{\phi}(\xi)d\xi.
\end{equation}

\smallskip
Some important time localized smoothing effects for the unitary groups $e^{it\partial_x^2}$ and $e^{-t\partial_x^3}$ can be found in \cite{KPV}. More specifically, we have the following estimates:
\begin{align*}
 &\|\psi(t)e^{it\partial_x^2}\phi(x)\|_{\mathcal{C}\big(\mathbb{R}_x;\; H^{(2s+1)/4}(\mathbb{R}_t)\big)}\leq c\|\phi\|_{H^{s}(\mathbb{R})},\\
 &\|\psi(t) e^{-t\partial_x^3}\phi(x)\|_{\mathcal{C}\big(\mathbb{R}_x;\; H^{(k+1)/3}(\mathbb{R}_t)\big)}\leq c \|\phi\|_{H^k(\mathbb{R})}\\
 \intertext{and}
 &\|\psi(t) \partial_xe^{-t\partial_x^3}\phi(x)\|_{\mathcal{C}\big(\mathbb{R}_x;\; H^{k/3}(\mathbb{R}_t)\big)}\leq c \|\phi\|_{H^k(\mathbb{R})},
\end{align*}
where $\psi(t)$ is a localized smooth cutoff function. This smoothing effects suggest that for data in the IBVP (\ref{SK}) is natural to consider the  following hypothesis: we put
\begin{align*}
&\mathscr{H}^{s,k}_+:=H^s(\mathbb{R}^+)\times H^k(\mathbb{R}^+)\times H^{(2s+1)/4}(\mathbb{R}^+)\times  H^{(k+1)/3}(\mathbb{R}^+),\\
&\mathscr{H}^{s,k}_-:= H^s(\mathbb{R}^-)\times H^k(\mathbb{R}^-)\times H^{(2s+1)/4}(\mathbb{R}^+)\times  H^{(k+1)/3}(\mathbb{R}^+)\times H^{k/3}(\mathbb{R}^+), 
\end{align*}
and then we take $(u_0,v_0,f,g)$ belonging to the class
\begin{equation}\label{hipinicial1}
\begin{cases}
\mathscr{H}^{s,k}_+& \text{if}\; s, k <1/2,\medskip\\ 
\mathscr{H}^{s,k}_+\;\text{with}\; u_0(0)=f(0) & \text{if}\; 1/2<s<3/2\; \text{and}\; k <1/2\medskip,\\ 
\mathscr{H}^{s,k}_+\;\text{with}\; v_0(0)=g(0) & \text{if}\; s<1/2\; \text{and}\; 1/2<k<3/2\medskip,\\ 
\mathscr{H}^{s,k}_+\;\text{with}\; u_0(0)=f(0)\;\text{and}\; v_0(0)=g(0) & \text{if}\; 1/2<s, k <3/2.
\end{cases}
\end{equation}
Similarly,  for  the IBVP (\ref{SKe}) we consider $(u_0,v_0,f,g,h)$ belonging to the class
\begin{equation}\label{hipinicial2}
\begin{cases}
\mathscr{H}^{s,k}_-& \text{if}\; s, k <1/2,\medskip\\ 
\mathscr{H}^{s,k}_-\;\text{with}\; u_0(0)=f(0) & \text{if}\; 1/2<s<3/2\; \text{and}\; k <1/2\medskip,\\ 
\mathscr{H}^{s,k}_-\;\text{with}\; v_0(0)=g(0) & \text{if}\; s<1/2\; \text{and}\; 1/2<k <3/2\medskip,\\ 
\mathscr{H}^{s,k}_-\;\text{with}\; u_0(0)=f(0)\;\text{and}\; v_0(0)=g(0) & \text{if}\; 1/2<s, k <3/2.
\end{cases}
\end{equation}

Our goal in studying IBVPs \eqref{SK}--\eqref{hipinicial1} and \eqref{SKe}--\eqref{hipinicial2} is to develop a local theory for 
low regularity of the initial data. 

\subsection{Main results}
We are in position to enunciate the main results of this work. We begin by establishing the results regarding the IBVP \eqref{SK}, which reads as follows: 

\begin{theorem}[\textbf{Local theory for right half-line}]\label{teorema1} 
Let the index sets (see Figure \ref{Figura-1}), defined by  
	\begin{align*}
	&\mathcal{D}:=\Big\{(s,k)\in\mathbb{R}^2;\; 0\leq s<\tfrac12\;\; \text{and}\;\; \max\big\{-\tfrac{3}{4},\,s-1\big\}<k<\min\big\{4s-\tfrac12,\, \tfrac12\big\} \Big\},\\
	&\mathcal{D}_0:=\Big\{(s,k)\in\mathbb{R}^2;\; \tfrac12<s<1\;\; \text{and}\;\; s-1< k<\tfrac12 \Big\}
	\end{align*}
	and consider the IBVP \eqref{SK} with condition \eqref{hipinicial1}, where 
	$$(s,k)\in \mathcal{D}\quad  \text{if}\quad  \beta\neq 0\quad \text{and}\quad  (s,k)\in \mathcal{D}\cup \mathcal{D}_0\quad \text{if}\quad \beta= 0.$$
	Then, there exists a positive time
	$$
	T=T\big(\|u_0\|_{H^s(\R^+)},\|v_0\|_{H^k(\R^+)},\|f\|_{H^{(2s+1)/4}(\R^+)},\|g\|_{H^{(k+1)/3}(\R^+)}\big)
	$$ 
	and a local solution $\big(u(\cdot, t),v(\cdot, t)\big)$, in the distributional sense, such that
	\begin{equation}\label{imagem}
	\big(u(\cdot, t), v(\cdot, t)\big)\in \mathcal{C}\big([0,T];\; H^s(\mathbb{R}^+)\times H^k(\mathbb{R}^+)\big).
	\end{equation}
	Moreover, the map $(u_0,v_0,f,g)\longmapsto (u(\cdot, t),v(\cdot, t))$ is locally Lipschitz-continuous from the space given in \eqref{hipinicial1} into the class  
	$\mathcal{C}\big([0,T];\; H^s(\mathbb{R}^+)\times H^k(\mathbb{R}^+)\big)$. 
\end{theorem}
\begin{remark}
	Note that the best result achieved  in Theorem \ref{teorema1} is for data $(u_0,v_0)$ in the Sobolev space
	$L^2(\R^+) \times H^{-\frac{3}{4}+}(\R^+)$, that is the regularity obtained by Corcho and Linares in \cite{CL}, which can be seen as the almost sharp regularity when $\beta\neq0$, except for the end point $(s, k)=(0, -\frac{3}{4})$.
\end{remark}

On the other hand, when we consider higher regularities, the technique employed here imposes smallness assumptions on data for the KdV  component of the system. More precisely, we have the following local theory:

\begin{theorem}[\textbf{Local theory for right half-line with small data}]\label{teorema11}
	Let the index sets (see Figure \ref{Figura-1}), defined by 
	\begin{align*}
	&\widetilde{\mathcal{D}}:=\Big\{(s,k)\in\mathbb{R}^2;\; \tfrac14<s<\tfrac12\;\; \text{and}\;\; \tfrac12<k<\min\big\{4s-\tfrac12,s+\tfrac12\big\} \Big\}\\
	&\widetilde{\mathcal{D}}_0:=\Big\{(s,k)\in\mathbb{R}^2;\; \tfrac12<s<1\;\; \text{and}\;\; \tfrac12<k<s+1/2 \Big\},
	\end{align*}
	and consider the IBVP \eqref{SK} with condition \eqref{hipinicial1}, where
    $$(s,k)\in \widetilde{\mathcal{D}}\quad  \text{if}\quad   \beta\neq 0\quad \text{and}\quad 
    (s,k)\in \widetilde{\mathcal{D}}\cup \widetilde{\mathcal{D}}_0\quad  \text{if}\quad \beta= 0.$$
    Assume in addition that
	\begin{equation}\label{dadospequenosRHL}
	\|v_0\|_{H^k(\mathbb{R}^+)}+\|g\|_{H^{(k+1)/3}(\mathbb{R}^+)}\leq \delta
	\end{equation}
	with $\delta$  small enough. Then, there exists a positive time
	$$
	T=T\big(\|u_0\|_{H^s(\R^+)},\|v_0\|_{H^k(\R^+)},\|f\|_{H^{(2s+1)/4}(\R^+)},\|g\|_{H^{(k+1)/3}(\R^+)}\big)
	$$ 
	and a local solution $\big(u(\cdot, t),v(\cdot, t)\big)$, in the distributional sense, such that
	\begin{equation}\label{imagem2}
	\big(u(\cdot, t), v(\cdot, t)\big)\in \mathcal{C}\big([0,T];\; H^s(\mathbb{R}^+)\times H^k(\mathbb{R}^+)\big).
	\end{equation}
	Moreover, the map $(u_0,v_0,f,g)\longmapsto (u(\cdot, t), v(\cdot,t))$ is locally Lipschitz-continuous from the space given in \eqref{hipinicial1} into the class  
	$\mathcal{C}\big([0,T];\; H^s(\mathbb{R}^+)\times H^k(\mathbb{R}^+)\big)$. 
\end{theorem}

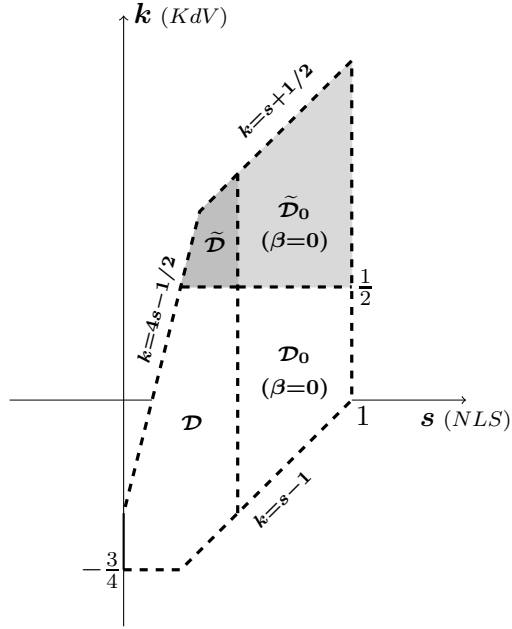
\begin{figure}[htp]
	\centering 
	\begin{tikzpicture}[scale=3]
	\draw[very thin](-0.5,0)--(1/8,0);
	\draw[very thin, ->] (1,0)--(1.5,0) node[below] {$\boldsymbol{s}\; {\scriptstyle (NLS)}$};
	\draw[->] (0,-1)--(0,1.7) node[right] {$\boldsymbol{k}\; {\scriptstyle (KdV)}$};
	\filldraw[color=gray!30](1,0.5)--(1,1.5)--(0.5,1)--(0.5,0.5)--(1,0.5);
	\filldraw[color=gray!50](0.5,0.5)--(0.5,1)--(1/3,5/6)--(0.25,0.5)--(0.5,0.5);
	\draw[very thick, dashed](0,-0.75)--(0.25,-0.75)--(1,0)--(1,1.5)--(0.5,1)--(0.5,0.5);
	\draw[very thick, dashed](0.5,1)--(1/3,5/6)--(0,-0.5);
	\draw[very thick, dashed](0.5,-0.5)--(0.5,0.5);
	\draw[very thick, dashed](0.5,0.5)--(1,0.5);
	\draw[very thick, dashed](0.5,0.5)--(0.25,1/2);
	\draw[thick](0,-0.5)--(0,-0.75);
	\node at (0.3,-0.1){$\boldsymbol{\scriptstyle \mathcal{D}}$};
	\node at (0.75,0.2){$\boldsymbol{\scriptstyle \mathcal{D}_0}$};
	\node at (0.75,0.05){$\boldsymbol{\scriptstyle (\beta=0)}$};
	\node at (0.75,0.85){$\boldsymbol{\scriptstyle \widetilde{\mathcal{D}}_0}$};
	\node at (0.75,0.7){$\boldsymbol{\scriptstyle (\beta=0)}$};
	\node at (0.4,0.7){$\boldsymbol{\scriptstyle \widetilde{\mathcal{D}}}$};
	\node at (1.05,-0.07){$1$};
	\node at (1.06,0.5){$\frac{1}{2}$};
	\node at (-0.1,-3/4){$-\frac{3}{4}$};
	\node at (0.67,1.32)[rotate=45]{\small{$\boldsymbol{\scriptstyle k=s+1/2}$}};
	\node at (0.15,0.4)[rotate=75.96375653]{\small{$\boldsymbol{\scriptstyle k=4s-1/2}$}};
	\node at (0.7,-0.44)[rotate=45]{\small{$\boldsymbol{\scriptstyle k=s-1}$}};
	\end{tikzpicture}
	\caption{{\small Regions of local well-posedness achieved in Theorems \ref{teorema1} and \ref{teorema11}. For the regions highlighted in gray the local theory is developed with smallness assumption on the data for the KdV-component of the system.}}
	\label{Figura-1}
\end{figure}

Similar results are obtained for the IBVP  (\ref{SKe}) posed on the left half-line, which are described as follows.

\begin{theorem}[\textbf{Local theory for left half-line}]\label{teorema2}
	Let the index sets (see Figure \ref{Figura-2}), defined by
	\begin{align*}
	&\mathcal{E}:=\Big\{(s,k)\in\mathbb{R}^2;\ \tfrac18<s<\tfrac12\ \text{and}\ 0\leq k<\min\big\{4s-\tfrac12,\tfrac12\big\}\Big\},\\
	&\mathcal{E}_0:=\Big\{(s,k)\in\mathbb{R}^2;\ \tfrac12<s<1\ \text{and}\ 0\leq k<\tfrac12\Big\}
	\end{align*}
	and consider the IBVP \eqref{SKe} with condition \eqref{hipinicial2}, where
	$$(s,k)\in \mathcal{E}\quad  \text{if}\quad \beta\neq 0,\quad \text{and}\quad 
	(s,k)\in \mathcal{E}\cup \mathcal{E}_0\quad  \text{if}\quad\beta= 0.$$
	Then, there exists a positive time
	$$
	T=T\big(\|u_0\|_{H^s(\R^-)},\|v_0\|_{H^k(\R^-)},\|f\|_{H^{(2s+1)/4}(\R^+)},\|g\|_{H^{(k+1)/3}(\R^+)}, \|h\|_{H^{k/3}(\R^+)}\big)
	$$ 
	and a local solution $\big(u(\cdot, t),v(\cdot, t)\big)$, in the distributional sense, such that
	 	\begin{equation}\label{imagem3}
	 	\big(u(\cdot, t), v(\cdot, t)\big)\in 	\mathcal{C}\big([0,T];\; H^{s}(\mathbb{R}^-)\times  H^{k}(\mathbb{R}^-)\big).
	 	\end{equation}
	Moreover, the map $(u_0,v_0,f,g,h)\longmapsto (u(t),v(t))$ is locally Lipschitz-continuous from the space given in \eqref{hipinicial2} into the class 
	$\mathcal{C}\big([0,T];\; H^{s}(\mathbb{R}^-)\times H^{k}(\mathbb{R}^-)\big)$.
\end{theorem}

Again, we need to imposes  smallness assumptions for data of the KdV component of the system im some regions of Sobolev indexes, now for low regularity besides.

\begin{theorem}[\textbf{Local theory for left half-line with small data}]\label{teorema22} 
Let the index sets (see Figure \ref{Figura-2}), defined by
	\begin{align*}
		&\tilde{\mathcal{E}}_{1}:=\Big\{(s,k)\in\mathbb{R}^2;\ 0<s<\tfrac12,\  \max\big\{-\tfrac34,s-1\big\}<k<\min\big\{0,4s-\tfrac12\big\}\Bigl\},\\
		&\tilde{\mathcal{E}}_{2}:=\Big\{(s,k)\in\mathbb{R}^2;\ \tfrac14<s<\tfrac12,\  \tfrac12<k<\min\big\{4s-\tfrac12,s+\tfrac12 \big\}\Big\},\\
		&\tilde{\mathcal{E}}_{1_0}:=\Big\{(s,k)\in\mathbb{R}^2;\ \tfrac12<s<1,\  s-1<k < 0\Big\},\\
		&\tilde{\mathcal{E}}_{2_0}:=\Big\{(s,k)\in\mathbb{R}^2;\ \tfrac12<s<1,\  \tfrac12 <k\leq s+\tfrac12\Big\}
	\end{align*}
	and consider the IBVP (\ref{SKe}) with condition \eqref{hipinicial2}, where
	 $$(s,k)\in \tilde{\mathcal{E}}_1\cup \tilde{\mathcal{E}}_2\quad \text{if}\quad \beta\neq 0\quad{and}\quad 
	 (s,k)\in \tilde{\mathcal{E}}_1\cup \tilde{\mathcal{E}}_2\cup \tilde{\mathcal{E}}_{1_0}\cup \tilde{\mathcal{E}}_{2_0}\quad \text{if}\quad \beta= 0.$$
	 Assume in addition that 
	\begin{equation}\label{dadospequenosLHL}
	\|v_0\|_{H^k(\mathbb{R}^-)}+\|g\|_{H^{(k+1)/3}(\mathbb{R}^+)}+\|h\|_{H^{k/3}(\mathbb{R}^+)}\leq \delta
	\end{equation}
    with $\delta$  small enough. Then, there exists a positive time
    	$$
    	T=T\big(\|u_0\|_{H^s(\R^-)},\|v_0\|_{H^k(\R^-)},\|f\|_{H^{(2s+1)/4}(\R^+)},\|g\|_{H^{(k+1)/3}(\R^+)}, \|h\|_{H^{k/3}(\R^+)}\big)
    	$$ 
    	and a local solution $\big(u(\cdot, t),v(\cdot, t)\big)$ , in the distributional sense, such that
	\begin{equation}\label{imagem4}
	\big(u(\cdot, t),v(\cdot, t)\big)\in \mathcal{C}\big([0,T];\; H^{s}(\mathbb{R}^-)\big)\times  \mathcal{C}\big([0,T];\; H^{k}(\mathbb{R}^-)\big).
	\end{equation}
	Moreover, the map $(u_0,v_0,f,g,h)\longmapsto (u(t),v(t))$ is locally Lipschitz-continuous from the space given in \eqref{hipinicial2} into the class  $\mathcal{C}\big([0,T];\; H^{s}(\mathbb{R}^-)\times H^{k}(\mathbb{R}^-)\big)$.
\end{theorem}

\begin{figure}[h]
	\begin{tikzpicture}[scale=3]
	\draw[->] (-0.5,0)--(1.5,0) node[below] {$\boldsymbol{s}\; {\scriptstyle (NLS)}$};
	\draw[->] (0,-1)--(0,1.7) node[right] {$\boldsymbol{k}\; {\scriptstyle(KdV)}$};
	\filldraw[color=gray!30](1,0.5)--(1,1.5)--(0.5,1)--(0.5,0.5)--(1,0.5);
	\filldraw[color=gray!50](0.5,0.5)--(0.5,1)--(1/3,5/6)--(0.25,0.5)--(0.5,0.5);
	\filldraw[color=gray!50](1/8,0)--(0.5,0)--(0.5,-0.5)--(0.25,-3/4)--(0,-3/4)--(0,-0.5)--(1/8,0);
	\filldraw[color=gray!30](0.5,0)--(1,0)--(0.5,-0.5)--(0.5,0);
	\draw[thick, dashed](0,-0.75)--(0.25,-0.75)--(1,0)--(1,1.5)--(0.5,1)--(0.5,0.5);
	\draw[thick, dashed](0.5,1)--(1/3,5/6)--(0,-0.5);
	\draw[thick, dashed](0.5,-0.5)--(0.5,0.5);
	\draw[thick, dashed](0.5,0.5)--(1,0.5);
	\draw[thick, dashed](0,-3/4)--(0,-0.5);
	\draw[thick, dashed](0.5,0.5)--(0.25,1/2);
	\draw[very thin] (1/8,0)--(1,0);
	\node at (0.35,0.25){$\boldsymbol{\scriptstyle \mathcal{E}}$};
	\node at (0.75,0.35){$\boldsymbol{\scriptstyle \mathcal{E}_0}$};
	\node at (0.75,0.2){$\boldsymbol{\scriptstyle (\beta=0)}$};
	\node at (0.38,0.7){$\boldsymbol{\scriptstyle \tilde{\mathcal{E}}_2}$};
	\node at (0.75,1.0){$\boldsymbol{\scriptstyle \tilde{\mathcal{E}}_{2_0}}$};
	\node at (0.75,0.8){$\boldsymbol{\scriptstyle (\beta=0)}$};
	\node at (0.35,-0.25){$\boldsymbol{\scriptstyle \tilde{\mathcal{E}}_1}$};
	\node at (0.63,-0.22){$\boldsymbol{\scriptstyle \tilde{\mathcal{E}}_{1_0}}$};
	\node at (0.7,-0.07){$\boldsymbol{\scriptstyle (\beta=0)}$};
	\node at (1.05,-0.07){$1$};
	\node at (1.05,0.5){$\frac{1}{2}$};
	\node at (-0.1,-3/4){$-\frac{3}{4}$};
	\node at (0.67,1.32)[rotate=45]{$\boldsymbol{\scriptstyle k=s+1/2}$};
	\node at (0.15,0.4)[rotate=75.96375653]{$\boldsymbol{\scriptstyle k=4s-1/2}$};
	\node at (0.71,-0.44)[rotate=45]{$\boldsymbol{\scriptstyle k=s-1}$};
	\end{tikzpicture}
	\caption{{\small Regions of local well-posedness achieved in Theorems \ref{teorema2} and \ref{teorema22}. For the regions highlighted in gray the local theory is developed with smallness assumption on the data for the KdV-component of the system.}}
	\label{Figura-2}
\end{figure}
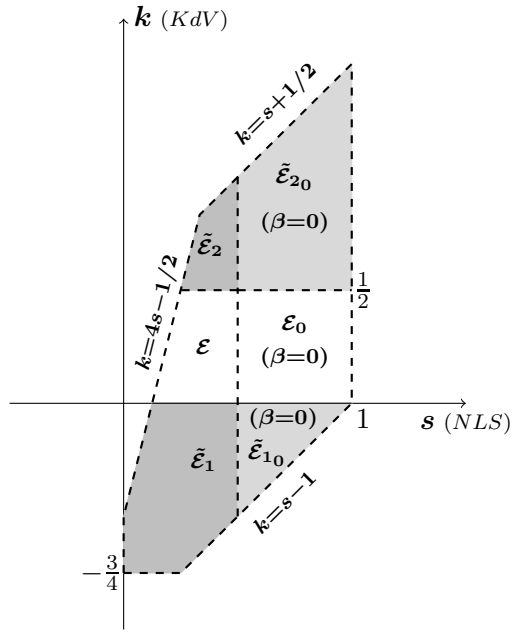
\smallskip

\begin{remark}
	Unlike the right-half line, for the left-half line the results achieved in Theorem \ref{teorema22} do not include the regularity $L^2(\R^-) \times H^{-\frac{3}{4}+}(\R^-)$. Indeed, in this case is needed the use of some modified Bourgain spaces where the bilinear estimate given in Proposition \ref{estimativau}  does not cover the regularity $s=0$ for the NLS component. 
\end{remark}

\subsection{Comments about uniqueness.} 
In this work we have uniqueness of the solutions in the sense of Kato (see [17]), that is,  uniqueness for a reformulation of the IBVPs \eqref{SK} and \eqref{SKe} as integral equations posed on the line $\R$, respectively. As there are many ways to
transform the IBVP into an integral equation, we do not have uniqueness in the strong sense. An approach to solve the question of unconditional uniqueness for IBVPs associated to nonlinear dispersive equations on the half-line was introduced by Bona, Sun and
Zhang in \cite{Bona}. These authors used the boundary integral method based on the Laplace transform to get the local theory for KdV equation on the half-line; for more details about this approach we refer the works \cite{Bona0} and \cite{Bona} concerning to the KdV equation and also the works \cite{Bona1} and  \cite{Tzirakis} concerning to the Schr\"odinger equation. We emphasize that in \cite{Bona} was introduced a concept of \textit{mild solutions} to solve the question of unconditional uniqueness for IBVP associated to the KdV equation on the half-line in the following systematic way: mild solutions are solutions that can be approximated by regular solutions, then they showed that the weak solutions obtained are also mild solutions and finally they proved the uniqueness of the mild solutions. In this work, we are not able to apply this approach, since we do not work in Sobolev spaces with high regularity, then we can not consider yet the existence of mild solutions with appropriated regularity. Local results in high regularity can not be obtained directly, since we have two difficulties: firstly, the boundary forcing operators are not well adapted for high regularities, as pointed out in the Lemmas \ref{edbf} and \ref{cof}; secondly, the bilinear estimates for the coupled terms also are very delicates even in high regularity. 
We are currently studying the global dynamic of solutions for the system \eqref{SK} and \eqref{SKe} in Sobolev spaces with high regularity.  In this forthcoming work we are using the boundary integral method combined with the bilinear estimates in high regularity and the evolution of the some classical conserved functional associated to the system.

\subsection{About the technique} The approach used to proof the main results is based on the arguments developed in \cite{Cavalcante}, \cite{CK},  \cite{Holmer} and \cite{Holmerkdv}. The main idea to solve the IBVP \eqref{SK} - \eqref{hipinicial1} is the construction of an auxiliary forced IVP in the line $\R$, analogous to the \eqref{SKpuro}; more precisely:
\begin{equation}\label{SKforcante}
	\begin{cases}
		iu_t+u_{xx}=\alpha uv+\beta |u|^2u+\mathcal{T}_1(x)h_1(t), & (x,t)\in\mathbb{R}\times(0,T),\\
		v_t+v_{xxx}+ vv_x=\gamma(|u|^2)_x+\mathcal{T}_2(x)h_2(t), & (x,t)\in\mathbb{R}\times(0,T),\\
		u(x,0)=\tilde{u}_0(x),\ v(x,0)=\tilde{v}_0(x), & x\in\R,
	\end{cases}
\end{equation}
where $\mathcal{T}_1$ e $\mathcal{T}_2$ are distributions supported in $\mathbb{R}^-$,  $\tilde{u}_0,\ \tilde{v}_0$ are nice extensions of  $u_0$ and $v_0$ in $\mathbb{R}$ and the
boundary forcing functions $h_1$, $h_2$ are selected to ensure that 
$$\tilde{u}(0,t)=f(t),\quad \text{and}\quad \tilde{v}(0,t)=g(t)$$ 
for all $t\in (0,T)$.

Upon constructing the solution $(\tilde{u}, \tilde{v})$ of IVP (\ref{SKforcante}), we obtain a distributional solution of IBVP (\ref{SK}) by restriction, that is $$(u,v)=\Big(\tilde{u}|_{\mathbb{R}^+\times (0,T)},\; \tilde{v}|_{\mathbb{R}^+\times (0,T)}\Big).$$
In the same way, to construct solutions  for the IBVP \eqref{SKe}-\eqref{hipinicial2} we need to solve the forced IVP
\begin{equation}\label{SKforcantee}
\begin{cases}
iu_t+u_{xx}=\alpha uv+\beta |u|^2u + \mathcal{T}_1(x)h_1(t),& (x,t)\in\mathbb{R}\times(0,T),\\
v_t+v_{xxx}+ vv_x=\gamma(|u|^2)_x + \mathcal{T}_2(x)h_2(t)+\mathcal{T}_3(x)h_3(t),& (x,t)\in\mathbb{R}\times(0,T),\\
u(x,0)=\tilde{u}_0(x),\ v(x,0)=\tilde{v}_0(x),& x\in\R,
\end{cases}
\end{equation}
where $\mathcal{T}_1$, $\mathcal{T}_2$ and $\mathcal{T}_3$ are distributions supported in $\mathbb{R}^+$,  $\tilde{u}_0,\ \tilde{v}_0$ are nice extensions of $u_0$ and $v_0$ in $\mathbb{R}$ and the  boundary forcing functions $h_1$, $h_2$ and $h_3$ are selected to ensure that 
$$\tilde{u}(0,t)=f(t),\quad \tilde{v}(0,t)=g(t),\quad \tilde{u}_x(0,t)=h(t)$$ 
for all $t\in (0,T)$.

\smallskip 
For both IVPs, \eqref{SKforcante} and \eqref{SKforcantee}, the solution is constructed using the classical Fourier restriction norm method used in \cite{CL, YW}  and the inversion of a Riemann-Liouville fractional integration operator. An important point in this work is that the Boundary forcing operators force us to work with Bourgain spaces $X^{s,b}$ and $Y^{k,b}$, associated for the Schr\"odinger and KdV equations, with $b<\frac{1}{2}$, then we need to obtain the nonlinear estimates for $b<\frac{1}{2}$, while in \cite{CL} and \cite{YW} the nonlinear estimates were obtained for $b>\frac{1}{2}$.

\begin{remark}
Unlike the single NLS and KdV equations, systems \eqref{SKforcante} and \eqref{SKforcantee} are not invariant under scaling, nevertheless we eliminate the smallness conditions on the initial data making  some refined localized time estimates. This was only possible for regularity of the initial data in the regions $\mathcal{D}$, $\mathcal{D}_0$, $\mathcal{E}$ and $\mathcal{E}_0$ indicated in figures \ref{Figura-1} and \ref{Figura-2}. For the remaining regularities where the local theory is established we were only able to remove the smallness assumption on the initial datum of the NLS component, the smallness assumption on the initial datum of the KdV component being still necessary. 
\end{remark}

\subsection{Structure of the paper}
This work is organized as follow: in the next section, we discuss some notations, introduce some important function spaces and review the definition and basic properties of the Riemann-Liouville fractional integral. Sections \ref{sectionlinear} and \ref{sectionduhamel} are devoted to summarize preliminary results about the free propagators of the single equations and also in these sections are established some useful estimates for the Duhamel boundary forcing operators and for the classes of the Duhamel boundary forcing operators associated to linear Schr\"odinger and KdV equations. Furthermore, similar estimates are established for the inhomogeneous solution operators associated to both equations. In Section \ref{nonlinearestimates}, the crucial bilinear estimates will be deduced and, finally, in Section \ref{prova-teoremas} is exhibited the proof of the main results about the local theory established in this work for the IBVPs \eqref{SK} and \eqref{SKe}.

\subsection*{Acknowledgments} 
M. Cavalcante  wishes to thank the Centro de Modelamiento Mate\-m\'atico - CMM  from Santiago de Chile, for the financial support and the  good scientific infrastructure
that allowed to conclude the final details of this paper. 

\section{\textbf{Preliminaries}}

Here we introduce some notations and function spaces as well as the  Riemann-Liouville fractional integral operator.

We put $\mathbb{R}^*=\mathbb{R}\setminus \{0\}.$ Throughout the paper the characteristic function of an arbitrary set $A$ is denoted by $\chi_{A}$ and also we fix a cutoff function $\psi \in \mathcal{C}_0^{\infty}(\mathbb{R})$ such that
$$
\psi(t)=
\begin{cases}
1 & \text{if}\; |t|\le 1,\\
0 & \text{if}\; |t|\ge 2,
\end{cases}
$$ 
and for $\delta>0$ we denote $\psi_{\delta}(t)=\frac{1}{\delta}\psi(\frac{t}{\delta})$.

For any real number we put $\langle x \rangle:=1+|x|$ and $f(x,y) \lesssim g(x,y)$ means that there is a constant $c$ such that 
$$f(x,y) \leq c g(x,y)\quad  \text{for all}\quad  (x, y)\in \R^2.$$ 

The classical Schwartz's space is denoted by $\mathscr{S}(\mathbb{R}^n)$ and the spcaes of tempered distributions is denoted by $\mathscr{S}'(\mathbb{R}^n)$. Given a function $\phi\in \mathscr{S}(\mathbb{R})$,\;  $\displaystyle \hat{\phi}(\xi)=\int_{\R} e^{-i\xi x}\phi(x)dx$ denotes the Fourier transform of $\phi$. For $u\in \mathscr{S}(\mathbb{R}^2)$, 
$$\hat{u}(\xi,\tau)=\displaystyle \iint_{\R^2}e^{-i(\xi x+\tau t)}u(x,t)dxdt$$
 denotes its space-time Fourier transform, $\mathscr{F}_{x}u(\xi,t)$  its space Fourier transform and $\mathscr{F}_{t}u(x,\tau)$ its time Fourier transform. 

We define $(\tau-i0)^{\alpha}=\lim\limits_{\gamma\rightarrow 0^{-}}(\tau-\gamma i)^{\alpha}$ in the sense of distributions.

\subsection{Function Spaces}
For $s\geq 0$ we say that $\phi \in H^s(\mathbb{R}^+)$ if exists $\tilde{\phi}\in H^s(\mathbb{R})$ such that 
$\phi=\tilde{\phi}|_{\R+}$.  In this case we set $\|\phi\|_{H^s(\mathbb{R}^+)}:=\inf\limits_{\tilde{\phi}}\|\tilde{\phi}\|_{H^{s}(\mathbb{R})}$. For $s\geq 0$ define $$H_0^s(\mathbb{R}^+)=\Big\{\phi \in H^{s}(\mathbb{R}^+);\,\text{supp} (\phi) \subset[0,+\infty) \Big\}.$$ For $s<0$, define $H^s(\mathbb{R}^+)$ and $H_0^s(\mathbb{R}^+)$  as the dual space of $H_0^{-s}(\mathbb{R}^+)$ and  $H^{-s}(\mathbb{R}^+)$, respectively. 

Define 
$$C_0^{\infty}(\mathbb{R}^+)=\Big\{\phi\in C^{\infty}(\mathbb{R});\, \text{supp}(\phi) \subset [0,+\infty)\Big\}$$
and $C_{0,c}^{\infty}(\mathbb{R}^+)$ as those members of $C_0^{\infty}(\mathbb{R}^+)$ with compact support and also recall that $C_{0,c}^{\infty}(\mathbb{R}^+)$ is dense in $H_0^s(\mathbb{R}^+)$ for all $s\in \mathbb{R}$. 

Finally we observe that a definition for $H^s(\R^-)$ and $H_0^s(\R^-)$ can be given analogous to that for $H^s(\R^+)$ and $H_0^s(\R^+)$.

The following results summarize useful properties of the Sobolev spaces and the proofs can be seen in \cite{CK}.
	
\begin{lemma}\label{sobolevh0}
Let $-\frac{1}{2}<s<\frac{1}{2}$. There is a constant $c_s$ such that
$$\|\chi_{(0,+\infty)}f\|_{H^s(\mathbb{R})}\leq c_s \|f\|_{H^s(\mathbb{R})},$$
for all $f\in H^s(\R).$
\end{lemma}

\begin{lemma}\label{sobolev0}
Let $0\leq s<\frac{1}{2}$. There is a constant $c_{s, \psi}$ such that
\begin{enumerate}
\item [(a)] $\|\psi f\|_{H^s(\mathbb{R})}\leq c_{s, \psi} \|f\|_{\dot{H}^{s}(\mathbb{R})}$,\medskip
\item [(b)] $\|\psi f\|_{\dot{H}^{-s}(\mathbb{R})}\leq c_{s, \psi} \|f\|_{H^{-s}(\mathbb{R})}$,
\end{enumerate}
for all $f$ in the spaces $\dot{H}^{s}(\mathbb{R})$ and $H^{-s}(\mathbb{R})$, respectively. 
\end{lemma}

\begin{lemma}\label{alta}
Let $\frac{1}{2}<s<\frac{3}{2}$. Then we have $H_0^s(\R^+)=\big\{f\in H^s(\R^+);f(0)=0\big\}$ and there is a constant $c_s$ such that 
$$\|\chi_{(0,+\infty)}f\|_{H_0^s(\R^+)}\leq c_s\|f\|_{H^s(\R^+)},$$
for all $f\in H^s(\R^+)$.
\end{lemma}

\begin{lemma}\label{cut}
Let $f\in  H_0^s(\mathbb{R}^+)$ with $s\in \R$. Then, there exists  a constant $c_{s, \psi}$ such that 
$$\|\psi f\|_{H_0^s(\mathbb{R}^+)}\leq c_{s, \psi} \|f\|_{H_0^s(\mathbb{R}^+)}.$$
\end{lemma}

Now consider the dispersive IVP of the form
 \begin{equation}\label{Abstrato}
 \begin{cases}
 iw_t-\phi(-i\partial_x)w=F(w)& \text{for}\quad (x,t)\in\mathbb{R}\times(0,T),\\
 w(x,0)=w_0(x)&\text{for}\quad x\in\mathbb{R},
 \end{cases}
 \end{equation}
 where $F$ is a nonlinear function, $\phi$ is a measurable real-valued function and  $\phi(-i\partial_x)$ is the multiplier operator by $\phi(\xi)$ via Fourier transform, that is,
 $$\big[\phi(-i\partial_x)w\big]^{\wedge}(\xi):= \phi(\xi)\widehat{w}(\xi).$$
 The corresponding integral formulation is given by
 \begin{equation*}
 w(t)=W_{\phi}(t)w_0-i\int_0^tW_{\phi}(t-t')F(w(t'))dt',
 \end{equation*}
 where $W_{\phi}(t)=e^{-it\phi(-i\partial_x)}$ is the unitary group that solves the linear part of (\ref{Abstrato}). We denote by 
 $X^{s,b}(\phi)$ the so called Bourgain space associated to \eqref{Abstrato}; more precisely, $X^{s,b}(\phi)$  is the completion of $\mathscr{S}'(\mathbb{R}^2)$ with respect to the norm
 \begin{equation*}\label{Bourgain-norm}
 \begin{split}
 \|w\|_{X^{s,b}(\phi)}&=\|W_{\phi}(-t)w \|_{H_t^b(\mathbb{R}:H_x^s(\mathbb{R}))}\\ &=\|\langle\xi\rangle^{s}\langle\tau\rangle^{b}\mathscr{F}(e^{it\phi(-i\partial_x)}w(\xi,\tau)\|_{L_{\tau}^2L^2_{\xi}}\\
 &=\!\|\langle\xi\rangle^s\langle\tau+\phi(\xi)\rangle^b\hat{w}(\xi,\tau)\rangle \|_{L_{\tau}^2L^2_{\xi}}.
 \end{split}
 \end{equation*}

Ginibre, Tsutsumi and Velo in \cite{GTV} while establishing local well-posedness results for the Zakharov system showed the following important estimate:
  \begin{lemma}\label{T}
 	Let $-\frac{1}{2}< b'< b\leq 0$\, or\, $0\leq b'<b<\frac{1}{2}$, $w\in X^{s,b}(\phi)$ and $s\in \mathbb{R}$. Then
 	\begin{equation*}
 	\|\psi_{T}w\|_{ X^{s,b'}(\phi)}\leq c T^{b-b'}\|w\|_{ X^{s,b}(\phi)}.
 	\end{equation*}
 \end{lemma}

Here we work with the Bourgain spaces associated to the Schr\"odinger and Airy groups and for these spaces we use, respectively, the following notation: $X^{s,b}:=X^{s,b}(\xi^2)$ and $Y^{s,b}:=X^{s,b}(-\xi^3)$.

To obtain our results we also need define the following auxiliary modified Bougain spaces. Let $W^{s,b}$, $U^{s,b}$ and $V^{\alpha}$ the completion of $S'(\R^2)$ with respect to the norms:
\begin{align*}
&\|w\|_{W^{s,b}}=\left(\int\int \langle \tau\rangle^{s} \langle \tau+\xi^2\rangle^{2b} |\widehat{w}(\xi,\tau)|^2d\xi d\tau\right)^{\frac{1}{2}},\\
&\|w\|_{U^{s,b}}=\left(\int\int \langle \tau\rangle^{2s/3} \langle \tau-\xi^3\rangle^{2b} |\widehat{w}(\xi,\tau)|^2d\xi d\tau\right)^{\frac{1}{2}}\\\intertext{and}
&\|w\|_{V^{\alpha}}=\left(\int\int  \langle \tau\rangle^{2\alpha} |\widehat{w}(\xi,\tau)|^2d\xi d\tau\right)^{\frac{1}{2}}.
\end{align*}

\subsection{Riemann-Liouville fractional integral}
The tempered distribution $\frac{t_+^{\alpha-1}}{\Gamma(\alpha)}$ is defined as a locally integrable function for Re $\alpha>0$, that is 
\begin{equation*}
\left \langle \frac{t_+^{\alpha-1}}{\Gamma(\alpha)},\ f \right \rangle:=\frac{1}{\Gamma(\alpha)}\int_0^{+\infty} t^{\alpha-1}f(t)dt.
\end{equation*}

For Re $\alpha>0$, integration by parts implies that
\begin{equation*}
\frac{t_+^{\alpha-1}}{\Gamma(\alpha)}=\partial_t^k\left( \frac{t_+^{\alpha+k-1}}{\Gamma(\alpha+k)}\right)
\end{equation*}
for all $k\in\mathbb{N}$. This expression allows to extend the definition, in the sense of distributions,  of $\frac{t_+^{\alpha-1}}{\Gamma(\alpha)}$ to all $\alpha \in \mathbb{C}$.

Integrating over an appropriate contour yields
\begin{equation}\label{transformada}
\left(\frac{t_+^{\alpha-1}}{\Gamma(\alpha)}\right)^{\widehat{}}(\tau)=e^{-\frac{1}{2}\pi\alpha}(\tau-i0)^{-\alpha},
\end{equation}
where $(\tau-i0)^{-\alpha}$ is the distributional limit. If $f\in C_0^{\infty}(\mathbb{R}^+)$, we define
\begin{equation*}
\mathcal{I}_{\alpha}f=\frac{t_+^{\alpha-1}}{\Gamma(\alpha)}*f.
\end{equation*}
Thus, when Re $\alpha>0$,
\begin{equation*}
\mathcal{I}_{\alpha}f(t)=\frac{1}{\Gamma(\alpha)}\int_0^t(t-s)^{\alpha-1}f(s)ds
\end{equation*} 
and notice that 
$$\mathcal{I}_0f=f,\quad  \mathcal{I}_1f(t)=\int_0^tf(s)ds,\quad \mathcal{I}_{-1}f=f'\quad  \text{and}\quad  \mathcal{I}_{\alpha}\mathcal{I}_{\beta}=\mathcal{I}_{\alpha+\beta}.$$

The following results state important properties of the Riemann-Liouville fractional integral operator. The proof of them can be found in \cite{Holmerkdv}.
\begin{lemma}
If $f\in C_0^{\infty}(\mathbb{R}^+)$, then $\mathcal{I}_{\alpha}f\in C_0^{\infty}(\mathbb{R}^+)$ 
for all $\alpha \in \mathbb{C}$.
\end{lemma}

\begin{lemma}\label{lio-lemaint}
If $0\leq \alpha <\infty$,\, $s\in \mathbb{R}$ and $\varphi \in C_0^{\infty}(\mathbb{R})$, then we have
\begin{align}
&\|\mathcal{I}_{-\alpha}h\|_{H_0^s(\mathbb{R}^+)}\leq c \|h\|_{H_0^{s+\alpha}(\mathbb{R}^+)},\label{lio}\\
\intertext{and}
&\|\varphi \mathcal{I}_{\alpha}h\|_{H_0^s(\mathbb{R}^+)}\leq c_{\varphi} \|h\|_{H_0^{s-\alpha}(\mathbb{R}^+)}.\label{lemaint}
\end{align}
\end{lemma}


\subsection{Elementary 1d-integral estimates}
The next inequalities will be used to estimate the nonlinear terms in Section \ref{nonlinearestimates}.
\begin{lemma}\label{lemanovo-lema2} The following integral estimates are valid:
\begin{align}
&\int_{-\infty}^{+\infty}\frac{dx}{\langle\alpha_0+\alpha_1x+x^2\rangle^b}\leq c\quad
\text{for all}\quad b>\frac{1}{2}\label{lemanovo}\\ \intertext{and}
&\int_{-\infty}^{+\infty}\frac{dx}{\langle\alpha_0+\alpha_1x+\alpha_2x^2+x^3\rangle^b}\leq c\quad
\text{for all}\quad  b>\frac{1}{3},\label{lema2}
\end{align}
where the constants $c$ only depend on $b$. 
\end{lemma}

\begin{proof}
	The details of the proof can be found in  \cite{BOP}. 
\end{proof}

\begin{lemma}\label{lemagtv1}
	Let $0\leq b_1,\ b_2<\frac{1}{2}$ with $b_{1}+b_{2}>\frac{1}{2}$. Then there exists a positive constant $c=c(b_{1}, b_{2})$ such that
	\begin{equation*}
	\int\frac{dy}{\langle y-\alpha\rangle^{2b_{1}}\langle y-\beta\rangle^{2b_{2}}}\leq \frac{c} {\langle \alpha-\beta\rangle^{2b_1+2b_2-1}},
	\end{equation*}
\end{lemma}
\begin{proof}
	See Lemma 4.2 in \cite{GTV}.
\end{proof}

We finish with another versions of the estimates given in Lemma \ref{lemagtv1}. 

\begin{lemma}\label{lema1-Holmerkdv}
The following integral estimate is valid:
\begin{equation*}
\int_{|x|<\beta}\frac{dx}{\langle x\rangle^{4b-1}|\alpha-x|^{\frac{1}{2}}}\leq c\frac{(1+\beta)^{2-4b}}{\langle\alpha\rangle^{\frac{1}{2}}}
\quad \text{for all}\quad b<\frac{1}{2},
\end{equation*}
where the constant $c$ only depends on $b$. 
\end{lemma}
\begin{proof}
See Lemma 5.13 in \cite{Holmerkdv}.
\end{proof}

\section{\textbf{Linear Versions for $\mathbb{R}^+$ and $\mathbb{R}^-$}}\label{sectionlinear}

In this section we shall obtain explicit solutions for the following linear versions of the IBVP associated to the liner part of the system; more precisely, we study the IBVP for linear Schr\"odinger equation:
\begin{equation}\label{jis1}
\begin{cases}
iu_t(x,t)+u_{xx}(x,t)=0& \text{for}\quad (x,t)\in\R^+\times (0,+\infty),\\
u(x,0)=u_0(x)& \text{for}\quad x\in \R^+,\\
u(0,t)=f(t)& \text{for}\quad t>0
\end{cases}
\end{equation}
and 
\begin{equation}\label{jis1-l}
	\begin{cases}
		iu_t(x,t)+u_{xx}(x,t)=0& \text{for}\quad (x,t)\in \R^-\times(0,+\infty),\\
		u(x,0)=u_0(x)& \text{for}\quad x\in \R^-,\\
		u(0,t)=f(t)& \text{for}\quad t>0.
	\end{cases}
\end{equation}
For linear version of KdV equation we study the following ones:
\begin{equation}\label{jis2}
\begin{cases}
v_t(x,t)+v_{xxx}(x,t)=0& \text{for}\quad  (x,t)\in \R^+ \times (0,+\infty),\\
v(x,0)=v_0(x)& \text{for}\quad x\in \R^+,\\
v(0,t)=f(t)& \text{for}\quad t>0
\end{cases}
\end{equation}
and 
\begin{equation}\label{jis3}
\begin{cases}
v_t(x,t)+v_{xxx}(x,t)=0& \text{for}\quad (x,t)\in \R^-\times(0,+\infty),\\
v(x,0)=v_0(x)& \text{for}\quad x\in \R^-,\\
v(0,t)=f(t)& \text{for}\quad t>0,\\
v_x(0,t)=h(t)& \text{for}\quad t>0.
\end{cases}
\end{equation}
For this purpose, we will use the approach given in \cite{Cavalcante}, \cite{CK}, \cite{Holmer} and \cite{Holmerkdv}, i.e., we solve forced problems in $\mathbb{R}$. More precisely, for the  Schr\"odinger case we solve the problems:
\begin{equation}\label{T2}
\begin{cases}
iu_t+u_{xx}=\mathcal{T}_1(x)h_1(t)& \text{for}\quad (x,t)\in \mathbb{R}\times(0,+\infty),\\
u(x,0)=\tilde{u}_0(x)& \text{for}\quad x\in\mathbb{R},
\end{cases}
\end{equation}
\begin{equation}\label{T2-l}
\begin{cases}
iu_t+u_{xx}=\mathcal{T}_2(x)h_2(t)& \text{for}\quad (x,t)\in \mathbb{R}\times(0, +\infty),\\
u(x,0)=\tilde{u}_0(x)& \text{for}\quad x\in\mathbb{R}
\end{cases}
\end{equation}
and for KdV the following ones:
\begin{equation}\label{T1}
\begin{cases}
v_t+v_{xxx}=\mathcal{T}_3(x)h_3(t)& \text{for}\quad (x,t)\in \mathbb{R}\times(0,+\infty),\\
v(x,0)=\tilde{v}_0(x)& \text{for}\quad x\in\mathbb{R},
\end{cases}
\end{equation}
\begin{equation}\label{T3}
\begin{cases}
v_t+ v_{xxx}=\mathcal{T}_4(x)h_4(t)+\mathcal{T}_5(x)h_5(t)& \text{for}\quad (x,t)\in \mathbb{R}\times(0,+\infty),\\
v(x,0)=\tilde{v}_0(x)& \text{for}\quad x\in\mathbb{R},
\end{cases}
\end{equation}
where $\tilde{u}_0$ and $\tilde{v}_0$ are nice extensions of $u_0$ and $v_0$ in $\mathbb{R}$, $\mathcal{T}_1$ and $\mathcal{T}_3$  are distributions supported in $\mathbb{R}^-$, 
$\mathcal{T}_2$, $\mathcal{T}_4$ and $\mathcal{T}_5$ are distributions supported in $\R^+$  and  $h_i\; (i=1,\dots,5)$ are suitable boundary forced functions.
Consequently, the solutions of IBVPs \eqref{jis1}, \eqref{jis1-l}, \eqref{jis2} and \eqref{jis3} are given, respectively, by the restriction of the solutions 
obtained for the problems \eqref{T2}, \eqref{T2-l}, \eqref{T1} and \eqref{T3}.

\begin{remark}
 The technique introduced  by Colliander and Kenig in \cite{CK} to solve the IVP \eqref{T1} in the context of generalized-KdV equation  was based on the delta-distribution $\delta_0(x)$.  Holmer \cite{Holmer} studied the IBVP \eqref{T2}  also using the delta-distribution in order to study the IBVP associated to nonlinear schr\"odinger equation and later, in \cite{Holmerkdv}, he solved the IVPS \eqref{T1} and \eqref{T3} by considering a more general family of distributions, namely $\frac{x_{-}^{\lambda-1}}{\Gamma(\lambda)}$ and $\frac{x_{+}^{\lambda-1}}{\Gamma(\lambda)}$ with $\lambda \in \mathbb{C}$, which  was necessary to obtain results in low regularity for the IBVP associated to the KdV equation on the half-line. Recently, following the Holmer's ideas, Cavalcante \cite{Cavalcante} solved the linear IBVP \eqref{jis1} by using the distributions $\frac{x_{-}^{\lambda-1}}{\Gamma(\lambda)}$, with $\lambda \in \mathbb{C}$, in order to solve IBVP associated to some quadratic nonlinear Schr\"odinger equations for data with low regularity.
\end{remark}

\subsection{Linear estimates for the free propagator of the Schr\"odinger equation}
We define the linear group 
$$e^{it\partial_x^2}:\mathscr{S}'(\mathbb{R})\rightarrow \mathscr{S}'(\mathbb{R})$$
associated to the linear Schr\"odinger equation as
\begin{equation*}
	e^{it\partial_x^2}\phi(x)=\big(e^{-it\xi^2}\widehat{\phi}(\xi)\big)^{\lor}(x),
\end{equation*}
so that 
\begin{equation}\label{linear}
\begin{cases}
		(i\partial_t+\partial_x^2)e^{it\partial_x^2}\phi(x) =0,& (x,t)\in\mathbb{R}\times\mathbb{R},\\
		e^{it\partial_x^2}\phi(x)\big|_{t=0}=\phi(x),& x\in\mathbb{R}.
\end{cases}
\end{equation}

\begin{lemma}\label{grupo}
	Let $s\in\mathbb{R}$ and  $0< b<1$. If $\phi\in H^s(\mathbb{R})$, then
	\begin{enumerate}
		\item[(a)] $\|e^{it\partial_x^2}\phi(x)\|_{\mathcal{C}\big(\mathbb{R}_t;\,H^s(\mathbb{R}_x^+)\big)}\leq c\|\phi\|_{H^s(\mathbb{R})}$ (\textbf{space traces}),\medskip 
		\item[(b)] $\|\psi(t) e^{it\partial_x^2}\phi(x)\|_{\mathcal{C}\big(\mathbb{R}_x;\,H^{(2s+1)/4}(\mathbb{R}_t)\big)}\leq c \|\phi\|_{H^s(\mathbb{R})}$ (\textbf{time traces}),\medskip
		\item [(c)]$\|\psi(t)e^{it\partial_x^2}\phi(x)\|_{X^{s,b}}\leq c\|\psi(t)\|_{H^1(\mathbb{R})} \|\phi\|_{H^s(\mathbb{R})}$
		(\textbf{Bourgain spaces}).
	\end{enumerate}
\end{lemma}

\begin{proof}
The assertion in (a) follows  from properties of group $e^{it\partial_x^2}$, (b) was obtained in \cite{KPV} and the proof of (c) can be found in \cite{GTV}.
\end{proof}

\subsection{The Duhamel boundary forcing operator associated to the linear Schr\"odinger equation}\label{section4}
We now introduce the Duhamel boundary forcing operator introduced in \cite{Holmer}. For $f\in C_0^{\infty}(\mathbb{R}^+)$, define the boundary forcing operator
\begin{eqnarray*}
	\mathcal{L}f(x,t)&=&2e^{i\frac{\pi}{4}}\int_0^te^{i(t-t')\partial_x^2}\delta_0(x)\mathcal{I}_{-\frac{1}{2}}f(t')dt'\\
	&=&\frac{1}{\sqrt{\pi}}\int_0^t(t-t')^{-\frac{1}{2}}e^{\frac{ix^2}{4(t-t')}}\mathcal{I}_{-\frac{1}{2}}f(t')dt'.
\end{eqnarray*}
The equivalence of the two equalities is clear from the formula

\begin{equation*}
\mathcal{F}_x\left( \frac{e^{-i \frac{\pi}{4}\text{sgn}\ a} }{2|a|^{1/2}\sqrt{\pi}}e^{\frac{ix^2}{4a}}\right)(\xi)=e^{-ia\xi^2},\ \forall\ a\in \mathbb{R}.
\end{equation*}

From this definition we see that
\begin{equation}\label{forçante}
\begin{cases}
		(i\partial_t+\partial_x^2)\mathcal{L}f(x,t)=2e^{i\frac{\pi}{4}}\delta_0(x)\mathcal{I}_{-\frac{1}{2}}f(t)
		& \text{for}\quad (x,t)\in \mathbb{R}\times(0,T),\\
		\mathcal{L}f(x,0)=0& \text{for}\quad  x\in \R,\\
		\mathcal{L}f(0,t)=f(t)& \text{for}\quad t\in(0,T).
\end{cases}
\end{equation}

The next result is concerning the properties of continuity of the functions $\mathcal{L}f(x,t)$. 
\begin{lemma}\label{continuidade}
	Let $f\in C_{0,c}^{\infty}(\mathbb{R}^+)$, then $\mathcal{L}f(x,t)$ verifies the following properties. 
	\begin{enumerate}
		\item[(a)] For a fixed time $t$, $\mathcal{L}f(x,t)$  is continuous with respect to the spatial variable $x\in \R$ and $\partial_x\mathcal{L}f(x,t)$ is continuous for all $x\neq 0$. Furthermore, 
		\begin{equation*}
		\lim_{x\rightarrow 0^{-}}\partial_x\mathcal{L}f(x,t)=e^{-i\frac{\pi}{4}}\mathcal{I}_{-1/2}f(t)
		\quad \text{and}\quad  
		\lim_{x\rightarrow 0^{+}}\partial_x\mathcal{L}f(x,t)=-e^{-i\frac{\pi}{4}}\mathcal{I}_{-1/2}f(t).
		\end{equation*}
	
		\item[(b)] For $N,k$ nonnegative integers and fixed $x$, $\partial_t^k\mathcal{L}f(x,t)$  is continuous in $t$ for all $t\in \mathbb{R^+}$. Moreover, we have pointwise estimates on the interval $[0,T]$ given by
		\begin{equation*}
		|\partial_t^k \mathcal{L}f(x,t)|+|\partial_x\mathcal{L}f(x,t)|\leq c\langle x\rangle^{-N},
		\end{equation*}
		where $c=c(f,N,k,T)$.
	\end{enumerate}
\end{lemma}

\begin{proof} 
	See Lemma 6.1 in \cite{Holmer}.
	\end{proof}

If we set $u(x,t)=e^{it\partial_x^2}\phi(x)+\mathcal{L}\big(f-e^{i\cdot \partial_x^2}\phi(\cdot)\big|_{x=0}\big)(x,t)$, then by  Lemma \ref{continuidade}-(a) $u(x,t)$ is continuous in $x$ and also $u(0,t)=f(t)$. Hence,  $u(x,t)$ solves, in the sense of distributions, the IBVP
\begin{equation}\label{forçante2}
	\begin{cases}
		iu_t(x,t)+u_{xx}(x,t)=0& \text{for}\quad (x,t)\in \mathbb{R}^*\times \mathbb{R},\\
		u(x,0)=\phi(x)& \text{for}\quad x\in\mathbb{R},\\
		u(0,t)=f(t)& \text{for}\quad t\in(0,T).
\end{cases}
\end{equation}
This would suffice to solve the linear analogue problem on the half-line.

\subsection{The Duhamel boundary forcing operator classes associated to linear Schr\"odinger equation}
Now we introduce the  Duhamel boundary forcing operator classes associated to the linear Schr\"odinger equation used in \cite{Cavalcante}. 

For $\lambda\in \mathbb{C}$ such that Re $\lambda>-2$  and $f\in C_0^{\infty}(\mathbb{R}^+)$ define
\begin{equation*}
\mathcal{L}_{+}^{\lambda}f(x,t)=\left[\frac{x_{-}^{\lambda-1}}{\Gamma(\lambda)}*\mathcal{L}\big(\mathcal{I}_{-\frac{\lambda}{2}}f\big)(\cdot,t)   \right](x),
\end{equation*}
 with $\frac{x_{-}^{\lambda-1}}{\Gamma(\lambda)}=\frac{(-x)_{+}^{\lambda-1}}{\Gamma(\lambda)}$, and 
\begin{equation*}
\mathcal{L}_{-}^{\lambda}f(x,t)=\left[\frac{x_{+}^{\lambda-1}}{\Gamma(\lambda)}*\mathcal{L}\big(\mathcal{I}_{-\frac{\lambda}{2}}f\big)(\cdot,t)   \right](x).
\end{equation*}
These definitions imply
\begin{equation}\label{forcings}
(i\partial_t+\partial_x^2)\mathcal{L}_{+}^\lambda f(x,t)=\frac{2e^{\frac{i\pi}{4}}}{\Gamma(\lambda)}x_{-}^{\lambda-1}\mathcal{I}_{-\frac{1}{2}-\frac{\lambda}{2}}f(t)
\end{equation}
and
\begin{equation}\label{forcings2}
(i\partial_t+\partial_x^2)\mathcal{L}_{-}^\lambda f(x,t)=\frac{2e^{\frac{i\pi}{4}}}{\Gamma(\lambda)}x_{+}^{\lambda-1}\mathcal{I}_{-\frac{1}{2}-\frac{\lambda}{2}}f(t),
\end{equation}
in the sense of distributions.

If Re $\lambda>0$, then
\begin{equation}\label{classe1}
\mathcal{L}_{+}^{\lambda}f(x,t)=\frac{1}{\Gamma(\lambda)}\int_{x}^{+\infty}(y-x)^{\lambda-1}\mathcal{L}\big(\mathcal{I}_{-\frac{\lambda}{2}}f\big)(y,t)dy,
\end{equation}
and
\begin{equation}\label{classe2}
\mathcal{L}_{-}^{\lambda}f(x,t)=\frac{1}{\Gamma(\lambda)}\int_{-\infty}^{x}(x-y)^{\lambda-1}\mathcal{L}\big(\mathcal{I}_{-\frac{\lambda}{2}}f\big)(y,t)dy,
\end{equation}

For Re $\lambda>-2$, using \eqref{forçante} we obtain
\begin{eqnarray}
\mathcal{L}_{+}^{\lambda}f(x,t)&=&\frac{1}{\Gamma(\lambda+2)}\int_{x}^{+\infty}(y-x)^{\lambda+1}\partial_y^2\mathcal{L}\big(\mathcal{I}_{-\frac{\lambda}{2}}f\big)(y,t)dy\nonumber\\
&=&-\int_{x}^{+\infty}\frac{(y-x)^{\lambda+1}}{\Gamma(\lambda+2)}\big(i\partial_{t}\mathcal{L}\mathcal{I}_{-\frac{\lambda}{2}}f\big)(y,t)dy\!\!+\!\!2e^{i\frac{\pi}{4}}\frac{x_{-}^{\lambda+1}}{\Gamma(\lambda+2)}\mathcal{I}_{-1/2-\lambda/2}f(t)\label{defnova}
\end{eqnarray}
and
\begin{eqnarray}
\mathcal{L}_{-}^{\lambda}f(x,t)&=&\frac{1}{\Gamma(\lambda+2)}\int_{-\infty}^{x}(x-y)^{\lambda+1}\partial_y^2\mathcal{L}\big(\mathcal{I}_{-\frac{\lambda}{2}}f\big)(y,t)dy\nonumber\\
&=&-\int_{-\infty}^{x}\frac{(x-y)^{\lambda+1}}{\Gamma(\lambda+2)}\big(i\partial_{t}\mathcal{L}\mathcal{I}_{-\frac{\lambda}{2}}f\big)(y,t)dy\!\!+\!\!2e^{i\frac{\pi}{4}}\frac{x_{+}^{\lambda+1}}{\Gamma(\lambda+2)}\mathcal{I}_{-1/2-\lambda/2}f(t)\label{defnova2}.
\end{eqnarray}

Notice that $\frac{x_{\pm}^{\lambda-1}}{\Gamma(\lambda)}\bigg|_{\lambda=0}=\delta_0$, then $\mathcal{L}_{\pm}^{0}f(x,t)=\mathcal{L}f(x,t)$ and by \eqref{defnova} and \eqref{defnova2} we have  $\mathcal{L}_{\pm}^{-1}f(x,t)=\partial_x\mathcal{L}(\mathcal{I}_{1/2}f)(x,t)$.

From Lemma \ref{continuidade} it follows that   $\mathcal{L}_{\pm}^{\lambda}f(x,t)$ is well defined for $ \lambda>-2$ for $t\in [0,1]$. Moreover, the dominated convergence theorem and Lemma \ref{continuidade} imply that, for fixed $t\in [0,1]$ and $\text{Re}\,\lambda>-1$, the function  $\mathcal{L}^{\lambda}f(x,t)$ is continuous in $x$ for all $x\in \mathbb{R}$.

The next result establishes  the values of $\mathcal{L}_{\pm}^{\lambda}f(x,t)$ at $x=0$.

\begin{lemma}\label{lemma-lzero}
	If $\emph{Re}\ \lambda>-1$ and $f\in C^{\infty}_0(\R^+)$, then
	\begin{equation}\label{lzero}
	\mathcal{L}^{\lambda}_{\pm}f(0,t)=e^{i\frac{\lambda\pi}{4}}f(t).
	\end{equation}
\end{lemma}

We finish this section with some estimates for the Duhamel boundary forcing operators classes $\mathcal{L}^{\lambda}_{\pm}$.

\begin{lemma}\label{edbf}
Let  $s\in\mathbb{R}$ and $f\in C_0^{\infty}(\mathbb{R}^+)$. The following estimates are valid: \medskip
	\begin{enumerate} 
		\item[(a)] (\textbf{space traces}) $\|\mathcal{L}_{\pm}^{\lambda}f(x,t)\|_{\mathcal{C}\big(\mathbb{R}_t;\,H^s(\mathbb{R}_x^+)\big)}\leq c \|f\|_{H_0^{(2s+1)/4}(\mathbb{R}^+)}$ whenever $s-\frac{3}{2}<\lambda<\min\big\{s+\frac{1}{2},\, \frac{1}{2}\big\}$ and $\text{supp}\,(f) \subset [0,1]$. \medskip 
		\item[(b)](\textbf{time traces}) $\|\psi(t)\mathcal{L}_{\pm}^{\lambda}f(x,t)\|_{\mathcal{C}\big(\mathbb{R}_x;\,H_0^{(2s+1)/4}(\mathbb{R}_t^+)\big)}\leq c \|f\|_{H_0^{(2s+1)/4}(\mathbb{R}^+)}$ whenever $-1<\lambda<1$. \medskip 
		\item[(c)](\textbf{Bourgain spaces}) $\|\psi(t)\mathcal{L}_{\pm}^{\lambda}f(x,t)\|_{X^{s,b}}\leq \|f\|_{H_0^{(2s+1)/4}(\mathbb{R}^+)}$ whenever $s-\frac{1}{2}<\lambda<\min\big\{s+\frac{1}{2},\, \frac{1}{2}\big\}$  and  $b<\frac{1}{2}$.	
	\end{enumerate}
\end{lemma}

The proof of lemmas \ref{lemma-lzero} and \ref{edbf} can be seen in \cite{Cavalcante}.
	
\subsection{Linear group associated to the KdV equation}
The linear unitary  group $e^{-t\partial_x^3}:\mathscr{S}'(\mathbb{R})\rightarrow \mathscr{S}'(\mathbb{R})$ associated to the linear KdV equation is defined by
\begin{equation*}
e^{-t\partial_x^3}\phi(x)=\Big(e^{it\xi^3}\widehat{\phi}(\xi)\Big)^{\lor{}}(x),
\end{equation*}
so that
\begin{equation}\label{lineark}
\begin{cases}
(\partial_t+\partial_x^3)e^{-t\partial_x^3}\phi(x,t)=0& \text{for}\quad (x,t)\in \mathbb{R}\times\mathbb{R},\\
e^{-t\partial_x^3}(x,0)=\phi(x)&\text{for}\quad  x\in\mathbb{R}.
\end{cases} 
\end{equation}

The next estimates were proven in \cite{Holmerkdv}.

\begin{lemma}\label{grupok}
	Let $k\in\mathbb{R}$ and  $0< b<1$. If $\phi\in H^s(\mathbb{R})$, then we have \medskip
	\begin{enumerate}
		\item[(a)] $\|e^{-t\partial_x^3}\phi(x)\|_{\mathcal{C}\big(\mathbb{R}_t;\,H^k(\mathbb{R}_x)\big)}\leq c\|\phi\|_{H^k(\mathbb{R})}$ (\textbf{space traces}), \medskip
		\item[(b)] $\|\psi(t) e^{-t\partial_x^3}\phi(x)\|_{\mathcal{C}\big(\mathbb{R}_x;\,H^{(k+1)/3}(\mathbb{R}_t)\big)}\leq c \|\phi\|_{H^k(\mathbb{R})}$ (\textbf{time traces}), \medskip 
		\item[(c)] $\|\psi(t) \partial_xe^{-t\partial_x^3}\phi(x)\|_{\mathcal{C}\big(\mathbb{R}_x;\,H^{k/3}(\mathbb{R}_t)\big)}\leq c \|\phi\|_{H^k(\mathbb{R})}$
		(\textbf{derivative time traces}),\medskip 
		\item [(d)] $\|\psi(t)e^{-t\partial_x^3}\phi(x)\|_{Y^{k,b}\cap V^{\alpha}}\leq c \|\phi\|_{H^k(\mathbb{R})}$ (\textbf{Bourgain spaces}).
	\end{enumerate}
\end{lemma}

\begin{remark}
	The spaces $V^{\alpha}$ introduced in \cite{Holmerkdv}  give us  useful auxiliary norms of the classical Bourgain spaces in order to validate the nonlinear estimates associated to the KdV equation.  
\end{remark}

\subsection{The Duhamel boundary forcing operator associated to the linear KdV equation}\label{section4}
Now we give the properties of the Duhamel boundary forcing operator introduced By Colliander and Kenig in \cite{Holmerkdv}, that is
\begin{equation}\label{lk}
\begin{split}
\mathcal{V}g(x,t)&=3\int_0^te^{-(t-t')\partial_x^3}\delta_0(x)\mathcal{I}_{-2/3}g(t')dt'\\
&=3\int_0^t\int_{\mathbb{R}}e^{i(t-t')\xi^3}e^{ix\xi}d\xi\mathcal{I}_{-2/3}g(t')dt'\\
&=3\int_0^t\int_{\mathbb{R}}\frac{1}{(t-t')^{1/3}}e^{i\xi^3}e^{i x\xi/(t-t')^{1/3}}d\xi\mathcal{I}_{-2/3}f(t')dt'\\
&=3\int_0^t A\left(\frac{x}{(t-t')^{1/3}}\right)\frac{\mathcal{I}_{-2/3}g(t')}{(t-t')^{1/3}}dt',
\end{split}
\end{equation}
defined for all $g\in C_0^{\infty}(\mathbb{R}^+)$ and $A$ denotes the Airy function 
$$A(x)=\frac{1}{2\pi}\int_{\xi}e^{ix\xi}e^{i\xi^3}d\xi.$$

From the definition of $\mathcal{V}$  it follows that 
\begin{equation}\label{forcingk}
\begin{cases}
(\partial_t+\partial_x^3)\mathcal{V}g(x,t)=3\delta_0(x)\mathcal{I}_{-\frac{2}{3}}g(t)& \text{for}\quad (x,t)\in \mathbb{R}\times\mathbb{R},\\
\mathcal{V}g(x,0)=0& \text{for}\quad x\in\mathbb{R}.
\end{cases}
\end{equation}

The proof of the results exhibited in this section were shown in \cite{Holmerkdv}.

\begin{lemma}\label{lemacrb1}
	Let $g\in C_0^{\infty}(\mathbb{R}^+)$ and consider a fixed time  $t\in[0,1]$. Then, 
	\begin{enumerate}
	\item[(a)]the functions $\mathcal{V}g(\cdot,t)$ and $\partial_x\mathcal{V}g(\cdot,t)$ are continuous in $x$ for all $x\in\mathbb{R}$. Moreover, they satisfy the spatial decay bounds
    $$
	|\mathcal{V}g(x,t)|+|\partial_x\mathcal{V}g(x,t)|\leq c_k\|g\|_{H^{k+1}}\langle x\rangle^{-k}\quad \text{for all}\;  k \geq 0; 
    $$
	\item[(b)] the function $\partial_x^2\mathcal{V}g(x,t)$ is continuous in $x$ for all $x\neq 0$ and has a step discontinuity
	of size $3\mathcal{I}_{\frac{2}{3}}g(t)$ at $x=0$. Also, $\partial_x^2\mathcal{V}g(x,t)$ satisfies the spatial decay bounds
	$$
	|\partial_x^2\mathcal{V}g(x,t)|\leq c_k \|f\|_{H^{k+2}}\langle x\rangle^{-k} \quad \text{for all}\; k\geq 0.
	$$
\end{enumerate}
	\end{lemma}

Since $A(0)=\frac{1}{3\Gamma\big(\frac{2}{3}\big)}$ from  \eqref{lk} we have that $\mathcal{V}g(0,t)=g(t).$
Thus, if we set
\begin{equation}
v(x,t)=e^{-t\partial_x^3}\phi(x)+\mathcal{V}\big(g-e^{-\cdot\partial_x^3}\phi\big|_{x=0}\big)(x,t),
\end{equation}
then $u$ solves the linear problem 
\begin{equation}
\begin{cases}
v_t(x,t)+v_{xxx}(x,t)=0& \text{for}\quad (x,t)\in \R^*\times\mathbb{R},\\
v(x,0)=\phi(x)& \text{for}\quad x\in\mathbb{R},\\
v(0,t)=g(t)& \text{for}\quad t\in(0, +\infty),
\end{cases}
\end{equation}
in the sense of distributions, and then this would suffice to solve the IBVP on the right half-line associated to linear KdV equation.

For the linear IBVP on the left half-line associated to the KdV equation, which has
two boundary conditions, it is considered the second boundary forcing operator associated to the linear KdV equation:
\begin{equation}
\mathcal{V}^{-1}g(x,t)=\partial_x\mathcal{V}\mathcal{I}_{1/3}g(x,t)=3\int_0^tA'\left( \frac{x}{(t-t')^{1/3}}\right)\frac{\mathcal{I}_{-1/3}g(t')}{(t-t')^{2/3}}dt'.
\end{equation}
From Lemma \ref{lemacrb1}, for all $g\in C_0^{\infty}(\mathbb{R}^+)$ the function $\mathcal{V}^{-1}g(x,t)$ is continuous in $x$ on $x\in\mathbb{R}$; moreover using that $A'(0)=-\frac{1}{3\Gamma(\frac{1}{3})}$ we get the relation  $\mathcal{V}^{-1}g(0,t)=-g(t).$

Also, the definition of $\mathcal{V}^{-1}g(x,t)$ allows us to ensure that
\begin{equation}
\begin{cases}
(\partial_t+\partial_x^3)\mathcal{V}^{-1}g(x,t)=3\delta_0'(x)\mathcal{I}_{-\frac{1}{3}}g& \text{for}\quad (x,t)\in \mathbb{R}\times\mathbb{R},\\
\mathcal{V}^{-1}g(x,0)=0& \text{for}\quad x\in\mathbb{R},
\end{cases}
\end{equation}
in the sense of distributions.

Furthermore, Lemma \ref{lemacrb1} implies that the function $\partial_x\mathcal{V}f(x,t)$ is continuous in $x$ for all $x\in\mathbb{R}$ and, since $A'(0)=-\frac{1}{3\Gamma(\frac{1}{3})}$,
\begin{equation}
\partial_x\mathcal{V}g(0,t)=-\mathcal{I}_{-\frac{1}{3}}g(t).
\end{equation}
Also,  $\partial_x\mathcal{V}^{-1}g(x,t)=\partial_x^2\mathcal{V}\mathcal{I}_{\frac{1}{3}}g(x,t)$ is continuous in $x$ for $x\neq 0$ and has a step discontinuity of size $3\mathcal{I}_{-\frac{1}{3}}g(t)$ at $x=0$. Indeed,
\begin{eqnarray*}
	\lim_{x\rightarrow 0^{+}}\partial_x^2\mathcal{V}g(x,t)&=&-\int_{0}^{+\infty}\partial_y^3\mathcal{V}g(y,t)dy=\int_{0}^{+\infty}\partial_t\mathcal{V}g(y,t)dy\\
	&=&3\int_{0}^{+\infty}A(y)dy\int_0^t\partial_t\mathcal{I}_{-\frac{2}{3}}g(t')dt'=\mathcal{I}_{-\frac{2}{3}}g(t),
\end{eqnarray*}
then from Lemma \ref{lemacrb1} -(b) we have
\begin{equation*}
\lim_{x\rightarrow 0^{-}}\partial_x\mathcal{V}^{-1}g(x,t)=-2\mathcal{I}_{-\frac{1}{3}}g(t)\quad \text{and}\quad  \lim_{x\rightarrow 0^{+}}\partial_x\mathcal{V}^{-1}g(x,t)=\mathcal{I}_{-\frac{1}{3}}g(t).
\end{equation*}
On the other hand, given $h_1(t)$ and $h_2(t)$ belonging to $C_0^{\infty}(\mathbb{R}^+)$ we have the relations:
\begin{eqnarray*}
	\mathcal{V}h_1(0,t)+\mathcal{V}^{-1}h_2(0,t)&=&h_1(t)-h_2(t),\\
	\lim_{x\rightarrow 0^{-}}\mathcal{I}_{\frac{1}{3}}\partial_x(\mathcal{V}h_1(x,\cdot)+\partial_x\mathcal{V}^{-1}h_2(x,\cdot))(t)&=&-h_1(t)-2h_2(t),\\
	\lim_{x\rightarrow 0^{+}}\mathcal{I}_{\frac{1}{3}}\partial_x(\mathcal{V}h_1(x,\cdot)+\partial_x\mathcal{V}^{-1}
	h_2(x,\cdot))(t)&=&-h_1(t)+h_2(t).
\end{eqnarray*}

For given $v_0(x)$, $g(t)$ and  $h(t)$ we assigned

\[ \Bigg[\begin{array}{c}
h_1 \vspace{0.3cm}\\
h_2 \end{array} \Bigg]:=\frac{1}{3}
\Bigg[\begin{array}{cc}
2  & -1 \vspace{0.3cm}\\
-1 & -1 \end{array} \Bigg]
\Bigg[\begin{array}{c}
g-e^{-\cdot\partial_x^3}v_0|_{x=0}\vspace{0.3cm}\\
\mathcal{I}_{\frac{1}{3}}\big(h-\partial_xe^{-\cdot\partial_x^3}v_0|_{x=0}\big)\end{array} \Bigg].\]
So, taking  $v(x,t)=e^{-t\partial_x^3}v_0(x)+\mathcal{V}h_1(x,t)+\mathcal{V}^{-1}h_2(x,t)$ we get
\begin{equation}\label{deltaprime}
\begin{cases}
v_t(x,t)+v_{xxx}(x,t)=0& \text{for}\quad (x,t)\in \mathbb{R}^*\times\mathbb{R},\\
v(x,0)=v_0(x)& \text{for}\quad x\in\mathbb{R},\\
v(0,t)=g(t)& \text{for}\quad t\in \R ,\\
\lim\limits_{x \rightarrow 0^{-}}\partial_xv(x,t)=h(t)& \text{for}\quad t\in \R,
\end{cases}
\end{equation}
in the sense of distributions.
\subsection{The Duhamel Boundary Forcing Operator Classes associated to linear KdV equation}

Now, we define the generalization of operators $\mathcal{V}$ and $\mathcal{V}^{-1}$ given by Holmer \cite{Holmerkdv}. Consider $\lambda\in \mathbb{C}$ with 
$\text{Re}\,\lambda>-3$ and $g\in C_0^{\infty}(\mathbb{R}^+)$. Define the operators
\begin{equation*}
\mathcal{V}_{-}^{\lambda}g(x,t)=\left[\frac{x_+^{\lambda-1}}{\Gamma(\lambda)}*\mathcal{V}\big(\mathcal{I}_{-\frac{\lambda}{3}}g\big)(\cdot,t)   \right](x)
\end{equation*}
 and
\begin{equation*}
\mathcal{V}_{+}^{\lambda}g(x,t)=\left[\frac{x_-^{\lambda-1}}{\Gamma(\lambda)}*\mathcal{V}\big(\mathcal{I}_{-\frac{\lambda}{3}}g\big)(\cdot,t)   \right](x),
\end{equation*}
with $\frac{x_{-}^{\lambda-1}}{\Gamma(\lambda)}=e^{i\pi \lambda}\frac{(-x)_{+}^{\lambda-1}}{\Gamma(\lambda)}$. Then, 
using \eqref{forcingk} we have that
\begin{equation}\label{kdvlambda-}
(\partial_t+\partial_x^3)\mathcal{V}_{-}^{\lambda}g(x,t)=3\frac{x_{+}^{\lambda-1}}{\Gamma(\lambda)}\mathcal{I}_{-\frac{2}{3}-\frac{\lambda}{3}}g(t)
\end{equation}
and
\begin{equation}\label{kdvlambda+}
(\partial_t+\partial_x^3)\mathcal{V}_{+}^{\lambda}g(x,t)=3\frac{x_{-}^{\lambda-1}}{\Gamma(\lambda)}\mathcal{I}_{-\frac{2}{3}-\frac{\lambda}{3}}g(t).
\end{equation}

\begin{lemma}[\textbf{Spatial continuity and decay properties for} $\boldsymbol{\mathcal{V}_{\pm}^{\lambda}g(x,t)}$]\label{holmer1}
	Let $g\in C_0^{\infty}(\mathbb{R}^+)$ and fix $t\geq 0$. Then, we have
	\begin{equation*}
		\mathcal{V}_{\pm}^{-2}g=\partial_x^2\mathcal{V}\mathcal{I}_{\frac{2}{3}}g,\quad \mathcal{V}_{\pm}^{-1}g=\partial_x\mathcal{V}\mathcal{I}_{\frac{1}{3}}g\quad \text{and}\quad \mathcal{V}_{\pm}^{0}g=\mathcal{V}g.
	\end{equation*}
	Also, $\mathcal{V}_{\pm}^{-2}g(x,t)$ has a step discontinuity of size $3g(t)$ at $x=0$, otherwise for $x\neq 0$, $\mathcal{V}_{ \pm}^{-2}g(x,t)$ is continuous in $x$. For $\lambda>-2$, $\mathcal{V}_{\pm}^{\lambda}g(x,t)$ is continuous in $x$ for all $x\in\mathbb{R}$. For $-2\leq\lambda\leq 1$ and  $0\leq t\leq 1$, $\mathcal{V}_{\pm}^{\lambda}g(x,t)$ satisfies the following decay bounds:
		\begin{align*}
	&|\mathcal{V}_{-}^{\lambda}g(x,t)|\leq c_{m,\lambda,g}\langle x\rangle^{-m},\; \text{for all}\; x\leq 0\; \text{and} \;m\geq0,\\ 
	&|\mathcal{V}_{-}^{\lambda}g(x,t)|\leq c_{\lambda,g}\langle x\rangle^{\lambda-1}\; \text{for all}\;  x\geq 0.\\
	&|\mathcal{V}_{+}^{\lambda}g(x,t)|\leq c_{m,\lambda,g}\langle x\rangle^{-m},\; \text{for all}\; x\geq 0\; \text{and} \;m\geq0,\\ \intertext{and}
	&|\mathcal{V}_{+}^{\lambda}g(x,t)|\leq c_{\lambda,g}\langle x\rangle^{\lambda-1}\; \text{for all}\;  x\leq 0.
	\end{align*}
	
	\end{lemma}

\begin{lemma}[\textbf{Values of} $\boldsymbol{\mathcal{V}_{\pm}^{\lambda}g(x,t)}$ \textbf{at} $\boldsymbol{x=0}$]\label{holmer2}
	For $\emph{Re}\,\lambda>-2$ and $g\in C_0^{\infty}(\R^+)$ we have 
	\begin{align*}
	&\mathcal{V}_{-}^{\lambda}g(0,t)=2\sin \left(\frac{\pi}{3}\lambda+\frac{\pi}{6}\right)g(t)\\ 
	\intertext{and}
	&\mathcal{V}_{+}^{\lambda}g(0,t)=e^{i\pi\lambda}g(t).
	\end{align*}
\end{lemma}

Now, as in \cite{Holmer} we shall solve the IVP forced \eqref{T3} with the distributions 	$\frac{x_{+}^{\lambda}}{\Gamma(\lambda)}$ by the using of the classes $\mathcal{V}_{-}^{\lambda}$.
	
	Let $-1<\lambda_1,\lambda_2<1$, $h_1,h_2\in C_0^{\infty}(\mathbb{R}^+)$ and set 
	\begin{equation}\label{probleft}
	v(x,t)=\mathcal{V}_{ -}^{\lambda_1}h_1(x,t)+\mathcal{L}_{ -}^{\lambda_2}h_2(x,t).
	\end{equation}
	By Lemma \ref{holmer1}, $u(x,t)$ is continuous in $x$ and Lemma \ref{holmer2} implies
	\begin{align*}
	& v(0,t)=2 \text{sin}\left(\frac{\pi}{3}\lambda_1+\frac{\pi}{6}\right)h_1(t)+2 \text{sin}\left(\frac{\pi}{3}\lambda_1+\frac{\pi}{6}\right)h_2(t)\intertext{and}\\ 
	&v_x(x,t)=\mathcal{V}_{ -}^{\lambda_1-1}\mathcal{I}_{-\frac{1}{3}}h_1(t)+\mathcal{V}_{ -}^{\lambda_2-1}\mathcal{I}_{-\frac{1}{3}}h_2(t).
	\end{align*}
	
	Also, by Lemma \ref{holmer1}, $v_x(x,t)$ is continuous in $x$ and Lemma \ref{holmer2} gives us that
	\begin{equation}
	\mathcal{I}_{\frac13}v_x(0,t)=2 \text{sin}\left(\frac{\pi}{3}\lambda_1-\frac{\pi}{6}\right)h_1(t)+2 \text{sin}\left(\frac{\pi}{3}\lambda_1-\frac{\pi}{6}\right)h_2(t).
	\end{equation}
	Therefore, we have
	\begin{equation}\label{Determinant-2x2}
	\left[\begin{array}{c}
	v(0,t)\vspace{0.3cm}\\
	\mathcal{I}_{1/3}v_x(0,t)
	\end{array}\right]=2
	\left[\begin{array}{cc}
	\text{sin}\left(\frac{\pi}{3}\lambda_1+\frac{\pi}{6}\right)&
	\text{sin}\left(\frac{\pi}{3}\lambda_2+\frac{\pi}{6}\right)\vspace{0.3cm}\\
	\text{sin}\left(\frac{\pi}{3}\lambda_1-\frac{\pi}{6}\right)& \text{sin}\left(\frac{\pi}{3}\lambda_2-\frac{\pi}{6}\right)\end{array}\right]\left[\begin{array}{r}
	h_1(t)\vspace{0.3cm}\\
	h_2(t)
	\end{array}\right].
	\end{equation}
Notice that the determinant of the $2 \times 2$  matrix in \eqref{Determinant-2x2} is given by the expression $\sqrt{3}\,\text{sin}\left(\frac{\pi}{3}(\lambda_2-\lambda_1)\right)$, which is nonzero if $\lambda_1-\lambda_2\neq 3n$ for $n\in \mathbb{Z}$. Thus, for any $-1<\lambda_1,\lambda_2<1$, with $\lambda_1\neq \lambda_2$ we set
	\begin{equation*}
	\left[\begin{array}{c}
	h_1(t)\vspace{0.3cm}\\
	h_2(t)
	\end{array}\right]=A\left[\begin{array}{c}
	g(t)\vspace{0.3cm}\\
	\mathcal{I}_{1/3}h(t)
	\end{array}\right],
	\end{equation*}
	where
	\begin{equation}
	A=\frac{1}{2\sqrt{3}\,\text{sin}\left(\frac{\pi}{3}(\lambda_2-\lambda_1)\right)}\left[\begin{array}{cc}
	\text{sin}\left(\frac{\pi}{3}\lambda_2-\frac{\pi}{6}\right)& -\text{sin}\left(\frac{\pi}{3}\lambda_2+\frac{\pi}{6}\right)\vspace{0.3cm}\\
	-\text{sin}\left(\frac{\pi}{3}\lambda_1-\frac{\pi}{6}\right)& \text{sin}\left(\frac{\pi}{3}\lambda_1+\frac{\pi}{6}\right)\end{array}\right].
	\end{equation}
and hence $v(x,t)$ solves \eqref{T3}.
	
We finish this section with further useful estimates, proved in \cite{Holmerkdv}, for the operators $\mathcal{V}_{\pm}^{\lambda}$.
 
\begin{lemma}\label{cof}
Let $k\in\mathbb{R}$.  The following estimates are ensured:\medskip
	\begin{enumerate}
		\item[(a)] (\textbf{space traces}) $\|\mathcal{V}_{\pm}^{\lambda}g(x,t)\|_{\mathcal{C}\big(\mathbb{R}_t;\,H^k(\R_x)\big)}\leq c \|g\|_{H_0^{(k+1)/3}(\mathbb{R}^+)}$ for all $k-\frac{5}{2}<\lambda<k+\frac{1}{2}$, $\lambda<\frac{1}{2}$  and $\emph{supp}(g)\subset[0,1]$.\medskip

		\item[(b)] (\textbf{time traces})
		$\|\psi(t)\mathcal{V}_{\pm}^{\lambda}g(x,t)\|_{\mathcal{C}\big(\mathbb{R}_x;\,H_0^{(k+1)/3}(\mathbb{R}_t^+)\big)}\leq c \|g\|_{H_0^{(k+1)/3}(\mathbb{R}^+)}$ for all $-2<\lambda<1$.\medskip

		\item[(c)] (\textbf{derivative time traces})
		$\big \|\psi(t)\partial_x\mathcal{V}_{\pm}^{\lambda}g(x,t)\big \|_{\mathcal{C}\big(\mathbb{R}_x;\,H_0^{k/3}(\mathbb{R}_t^+)\big)}\leq c \|g\|_{H_0^{(k+1)/3}(\mathbb{R}^+)}$ for all $~{-1<\lambda<2}$.\medskip
	
		\item[(d)] (\textbf{Bourgain spaces}) $
			\big\|\psi(t)\mathcal{V}_{\pm}^{\lambda}g(x,t)\big\|_{Y^{k,b}\cap V^{\alpha}}\leq c \|g\|_{H_0^{(k+1)/3}(\mathbb{R}^+)}
	$ for all $k-1\leq \lambda<k+\frac{1}{2}$, $\lambda<\frac{1}{2}$, $\alpha\leq\frac{s-\lambda+2}{3}$ and  $0\leq b<\frac{1}{2}$.
	\end{enumerate}
\end{lemma}

\section{\textbf{The Duhamel Inhomogeneous Solution Operators}}\label{sectionduhamel}
The  Duhamel inhomogeneous solution operator $\mathcal{S}$ associated with Schr\"odinger equation is define by 
\begin{equation*}
\mathcal{S}w(x,t)=-i\int_0^te^{i(t-t')\partial_x^2}w(x,t')dt',
\end{equation*}
so that
\begin{equation}\label{nonlinears}
\begin{cases}
(i\partial_t+\partial_x^2)\mathcal{S}w(x,t) =w(x,t)&\text{for}\quad  (x,t)\in\mathbb{R}\times\mathbb{R},\\
\mathcal{S}w(x,0)=0& \text{for}\quad x\in\mathbb{R}.
\end{cases}
\end{equation}
The corresponding inhomogeneous solution operator $\mathcal{K}$ associated to the KdV equation is given by
\begin{equation*}
\mathcal{K}w(x,t)=\int_0^te^{-(t-t')\partial_x^3}w(x,t')dt',
\end{equation*}
thus we have 
\begin{equation}\label{DK}
\begin{cases}
(\partial_t+\partial_x^3)\mathcal{K}w(x,t) =w(x,t)&\text{for}\quad  (x,t)\in\mathbb{R}\times\mathbb{R},\\
\mathcal{K}w(x,0) =0 & \text{for}\quad x\in\mathbb{R}.
\end{cases}
\end{equation}

Now we summarize some useful estimates for the Duhamel inhomogeneous solution operators $\mathcal{S}$ and $\mathcal{K}$  that will be used later in the proof of the main results.

The following lemma is due to Cavalcante and its proof can be seen in \cite{Cavalcante}.

\begin{lemma}\label{duhamel}
Let $s\in \mathbb{R}$. The following estimates are ensured:\medskip
	\begin{enumerate}
		\item[(a)](\textbf{space traces})
		$\|\psi(t)\mathcal{S}w(x,t)\|_{\mathcal{C}\big(\mathbb{R}_t;\,H^s(\mathbb{R}_x)\big)}\leq c\|w\|_{X^{s,d_1}}$ whenever $-\frac{1}{2}<d_1<0$.\medskip
	
		\item[(b)](\textbf{time traces}) For all $-\frac{1}{2}<d_1<0$ we have that
		 \begin{equation*}
		\|\psi(t)\mathcal{S}w(x,t)\|_{\mathcal{C}\big(\mathbb{R}_x;H^{(2s+1)/4}(\mathbb{R}_t)\big)}\lesssim
		\begin{cases}
		\|w\|_{X^{s,d_1}}& \text{if}\; -\frac{1}{2}< s\leq\frac{1}{2},\medskip\\
		(\|w\|_{W^{s,d_1}}+\|w\|_{X^{s,d_1}})& \text{for all}\; s\in\mathbb{R}. 
		\end{cases}
		\end{equation*}
        \medskip
		
		\item[(c)](\textbf{Bourgain spaces estimates}) Let $-\frac{1}{2}<d_1\leq0\leq b\leq d_1+1$, then we have
		\begin{equation*}
		\|\psi(t)\mathcal{S}w(x,t)\|_{X^{s,b}}\leq c \|w\|_{X^{s,d_1}}.
		\end{equation*}
	\end{enumerate}
\end{lemma}

Similar results for KdV equation were obtained by Holmer in \cite{Holmerkdv}, which read as follows:

\begin{lemma}\label{duhamelk} For all $k\in\mathbb{R}$ we have the following estimates:\medskip
	\begin{enumerate}
		\item[(a)](\textbf{space traces}) Let $-\frac{1}{2}<d_2<0$, then
		\begin{equation*}
		\|\psi(t)\mathcal{K}w(x,t)\|_{\mathcal{C}\big(\mathbb{R}_t;\,H^k(\mathbb{R}_x)\big)}\leq c\|w\|_{Y^{k,d_2}}.
		\end{equation*}\medskip
		
		\item[(b)](\textbf{time traces})
		If $-\frac{1}{2}<d_2<0$, then
		\begin{equation*}
		\left\|\psi(t)\mathcal{K}w(x,t)\right\|_{\mathcal{C}\big(\mathbb{R}_x;\,H^{(k+1)/3}(\mathbb{R}_t)\big)}\lesssim \begin{cases}
		\|w\|_{Y^{k,d_2}} & \text{if}\; -1\leq k \leq \frac{1}{2},\\
	    \|w\|_{Y^{k,d_2}}+\|w\|_{U^{k,d_2}}& \text{for all}\;  k\in\mathbb{R}.
		\end{cases}
		\end{equation*}\medskip
		
		\item[(c)](\textbf{derivative time traces}) If $-\frac{1}{2}<d_2<0$, then
		\begin{equation*}
		\|\psi(t)\partial_x\mathcal{K}w(x,t)\|_{\mathcal{C}\big(\mathbb{R}_x;\,H_t^{k/3}(\mathbb{R})\big)}\lesssim
		\begin{cases}
		\|w\|_{Y^{k,d_2}}& \text{if}\;\;0\leq k \leq \frac{3}{2},\\
		\|w\|_{Y^{k,d_2}}+\|w\|_{U^{k,d_2}}& \text{for all}\; k\in\mathbb{R}.
		\end{cases}
		\end{equation*}
		
		\item[(d)](\textbf{Bourgain spaces estimates})
	    Let $0<b<\frac{1}{2}$ and $\alpha>1-b$, then
		\begin{equation*}
		\|\psi(t)\mathcal{K}w(x,t)\|_{Y^{k,b}\cap V^{\alpha}}\leq \|w\|_{Y^{k,-b}}.
		\end{equation*}
	\end{enumerate}
\end{lemma}

\begin{remark}
	We note that the  time-adapted Bourgain spaces $W^{s,d_1}$ and $U^{k,d_2}$ used in lemmas \ref{duhamel}-(b) and  \ref{duhamelk}-(c), respectively, are introduced in order to cover the full values of regularity $s$ and $k$. 
\end{remark}

\section{\textbf{Nonlinear Estimates}}\label{nonlinearestimates}

Now we prove the main estimates for the nonlinear terms that are needed in the proof of our main results. 

\subsection{Known nonlinear estimates} We begin by recall previous nonlinear estimates obtained in the works \cite{Bourgain1} and \cite{Holmerkdv}.

We begin with the trilinear estimate for the nonlinear term of the classical cubic-NLS equation, deduced by Bourgain in \cite{Bourgain1}.

\begin{lemma}\label{trilinear}
	Let $u_1$, $u_2$, $u_3 \in X^{s,b}$ with $\frac{3}{8}<b<\frac{1}{2}$ and $s\geq 0$. Then, for all  $0<a<\frac{1}{2}$ we have that
	\begin{equation*}
	\|u_1u_2\overline{u}_3\|_{X^{s,-a}}\leq c\|u_1\|_{X^{s,b}}\|u_2\|_{X^{s,b}}\|u_3\|_{X^{s,b}}.
	\end{equation*}
\end{lemma}

Next nonlinear estimates, in the context of the KdV equation, for $b<\frac{1}{2}$, was derived by Holmer in \cite{Holmerkdv}.
	\begin{lemma}\label{bilinear1}
		Let $v_1,\ v_2\in Y^{k,b}\cap V^{\alpha}$, with $k>-\frac{3}{4}$, $\alpha>\frac{1}{2}$ and $\max\big\{\frac{5}{12}-\frac{s}{9}, \frac{1}{4}-\frac{s}{3},\frac{3}{10}-\frac{s}{15},\frac{1}{4}\big\}<b<\frac{1}{2}$. Then we have
		\begin{equation}
		\big\|(v_1v_2)_x\big\|_{Y^{k,-b}}\leq c\|v_1\|_{Y^{k,b}\cap V^{\alpha}}\|v_2\|_{Y^{k,b}\cap V^{\alpha}}.
		\end{equation}
	\end{lemma}
	
		\begin{remark}
			The purpose of introducing the space $V^{\alpha}$ (low frequency correction factor) was done by Holmer in \cite{Holmerkdv} to validate the bilinear estimates above for $b<\frac{1}{2}$. Recall that the need to take $b<\frac{1}{2}$ appears in Lemma \ref{cof}-(d).
		\end{remark}

\subsection{Bilinear estimates for the coupling terms}
We finish this section deriving new bilinear estimates for the interaction terms of the system.
\begin{proposition}\label{acoplamento1}
	Let $s,k,a,b\in \mathbb{R}$ with $k-|s|> \max\{2-6b,\frac{5}{2}-9a\}$ and $\frac{7}{18}<2b-\tfrac12 \le a<b$. Then there exists a positive constant $c=c(s,k,a,b)$ such that 
	\begin{equation*}
	\|uv\|_{X^{s,-a}}\leq c \|u\|_{X^{s,b}}\|v\|_{Y^{k,b}}
	\end{equation*}
	 for any $u\in X^{s,b}$ and $v\in Y^{k,b}$.
\end{proposition}
\begin{proposition}\label{estimativaw}
	Let $s,k,a,b\in \mathbb{R}$, with $\frac{1}{2} < s\leq 2a$, $\frac{1}{3}<a<b<\frac{1}{2}$ and $k>s-2a$. Then, there exists a positive constant $c=c(s,k,a,b)$ such that
	\begin{equation*}
	\|uv\|_{W^{s,-a}}\leq c\|u\|_{X^{s,b}}\|v\|_{Y^{k,b}}
	\end{equation*}
	for any $u\in X^{s,b}$ and  $v\in Y^{k,b}$.
\end{proposition}
\begin{remark}
In the above estimate the hypothesis $s\leq 2a$ is sufficient to cover the set $s<1$. Moreover, is sufficient  $s>1/2$ 
since for the case $s\leq 1/2$ the implementation of the space $W^{s,b}$ is not needed in the estimate (b) of Lemma \ref{duhamel}.
\end{remark}

\begin{proposition}\label{acoplamento2}
	Let $s,k,a,b\in \mathbb{R}$ with $s\geq 0$, $$k\leq\min\left\{s+6b+3a-\frac{7}{2},\ s+3b-1,\ 4s+2a-\frac{3}{2},\ 4s+3a+6b-\frac{7}{2}\right\}$$ and  $\frac38 <a\leq b<\frac{1}{2}$. Then there exists a positive constant $c=c(a,b,s,k)$ such that 
		\begin{equation*}
	\|(u_1\bar{u}_2)_x\|_{Y^{k,-a}}\leq c\|u_1\|_{X^{s,b}}\|u_2\|_{X^{s,b}}
	\end{equation*}
	for any $u_1,u_2\in X^{s,b}$.
\end{proposition} 

\begin{proposition}\label{estimativau}
Let $s,k,a,b\in \mathbb{R}$ verifying one of the following conditions:
\begin{enumerate}
\item[(a)] $\frac{1}{4}<b<\frac{1}{2}$,\; $s>\frac{1}{4}$\; 
and\;\;$0\leq k\leq \min\big\{3a,\,2s+6b+3a-\frac{7}{2}\big\}$,\medskip
\medskip

\item[(b)] $\frac{1}{4}<b<\frac{1}{2}$,\; $1-2b <s\leq 3a-1/2$ and  $k\leq 0$.
\end{enumerate}

Then, there exists a constant $c=c(k,s,b,a)$
	\begin{equation*}
	\|(u_1\bar{u}_2)_x\|_{U^{k,-a}}\leq c\|u_1\|_{X^{s,b}}\|u_2\|_{X^{s,b}}
	\end{equation*}
for any $u_1,u_2\in X^{s,b}$. 
\end{proposition}

\begin{remark}
	The estimate above with the condition in (b) is necessary to obtain the results of Theorems \ref{teorema2} and \ref{teorema22} concerning to the  IBVP \ref{SKe}. Since $2b+s>1$ and $b<1/2$ the case $s=0$  is not covered.
\end{remark}

Now we show the proofs of the statements above. We follow closely the arguments in \cite{CL}.

\subsection{Proof of Proposition \ref{acoplamento1}} Arguing as in the proof of Lemma 3.1 of \cite{CL} we need to show the following boundedness:
\begin{align}
&\left\|w_1(\xi,\tau)\!:=\!\frac{\langle \xi\rangle^{2s}}{\langle \tau+\xi^2\rangle^{2a}}\int\int \frac{\chi_{\mathcal{R}_1}d\tau_1 d\xi_1}{\langle \tau_1-\xi_1^3\rangle^{2b}\langle \xi_1\rangle^{2k}\langle \xi-\xi_1\rangle^{2s}\langle \tau-\tau_1 +(\xi-\xi_1)^2 \rangle^{2b}}\right\|_{L_{\xi,\tau}^{\infty}}\!\leq\! c,\label{w1}\\
&\left\|w_2(\xi_1,\tau_1)\!\!:=\!\!\frac{1}{\langle \xi_1\rangle^{2k}\langle \tau_1-\xi_1^3\rangle^{2b}}\int\int \frac{\langle\xi\rangle^{2s}\chi_{\mathcal{R}_2}d\tau d\xi}{\langle \tau+\xi^2\rangle^{2a}\langle \tau-\tau_1 +(\xi-\xi_1)^2 \rangle^{2b}\langle \xi-\xi_1\rangle^{2s}}\right\|_{L_{\xi_1\tau_1}^{\infty}}\!\!\leq\!\! c,\label{w2}\\
&\left\|w_3(\xi_2,\tau_2)\!\!:=\!\!\frac{1}{\langle \xi_2\rangle^{2s}\langle \tau_2-\xi_2^2\rangle^{2b}}\int\int \frac{\langle\xi_1-\xi_2\rangle^{2s}\chi_{\mathcal{R}_3}d\tau_1 d\xi_1}{\langle \tau_1-\tau_2+(\xi_1-\xi_2)^2\rangle^{2a}\langle \tau_1 -\xi_1^3 \rangle^{2b}\langle \xi_1\rangle^{2k}}\right\|_{L_{\xi_2\tau_2}^{\infty}}\leq c,\label{w3}
\end{align}
for suitable regions $\mathcal{R}_j\subset \mathbb{R}^4\; (j=1,2,3)$ such that $\mathcal{R}_1\cup \mathcal{R}_2\cup \mathcal{R}_3=\mathbb{R}^4$. Here we refine the regions 
given in \cite{CL}. More precisely, first we define the set 
$$\mathcal{A}:=\left\{(\xi,\xi_1,\tau,\tau_1)\in\mathbb{R}^4;\; |\xi_1|> 2\; \text{and}\;\ |\xi_1^2-\xi_1+2\xi|\geq \tfrac{1}{2}|\xi_1|^2\right\}$$
and then we put
$$\mathcal{R}_1:=\mathcal{R}_{1_1}\cup \mathcal{R}_{1_2}\cup\mathcal{R}_{1_3},$$
with
\begin{align*}
&\mathcal{R}_{1_1}:=\left\{(\xi,\xi_1,\tau,\tau_1)\in\mathbb{R}^4;\; |\xi_1|\leq 2\right\},\\
&\mathcal{R}_{1_2}:=\left\{(\xi,\xi_1,\tau,\tau_1)\in\mathbb{R}^4;\; |\xi_1|> 2\, \text{and}\, \ |\xi_1^2-\xi_1+2\xi|\leq \tfrac{1}{2}|\xi_1|^2\right\},\\
&\mathcal{R}_{1_3}:=\left\{(\xi,\xi_1,\tau,\tau_1)\in\mathcal{A};\; 
\max \{|\tau_1-\xi_1^3|,|\tau-\tau_1+(\xi-\xi_1)^2|,|\tau+\xi^2|\}=|\tau+\xi^2|\right\}.
\end{align*}
The two remaining regions are defined as follows:
\begin{align*}
 &\mathcal{R}_2:=\left\{(\xi,\xi_1,\tau,\tau_1)\in\mathcal{A};\; \max \big\{|\tau_1-\xi_1^3|,|\tau-\tau_1+(\xi-\xi_1)^2|,|\tau+\xi^2|\big\}=|\tau_1-\xi_1^3|\right\},\\
 &\mathcal{R}_3:=\left\{(\xi,\xi_1,\tau,\tau_1)\in\mathcal{A};\; \max \big\{|\tau_1-\xi_1^3|,|\tau-\tau_1+(\xi-\xi_1)^2|,|\tau+\xi^2|\big\}=|\tau-\tau_1+(\xi-\xi_1)^2|\right\}.
 \end{align*}

Before starting with the estimates we observe that obviously $\mathcal{R}_1\cup \mathcal{R}_2\cup \mathcal{R}_3=\mathbb{R}^4$ and notice that the points of the set $\mathcal{R}_{1_2}$ verify
the following inequality:
\begin{equation}\label{desigualdade-R12}
\tfrac{1}{2}|\xi_1|^2\leq |3\xi_1^2-2\xi_1+2\xi|.
\end{equation}
Indeed, if $(\xi,\xi_1,\tau,\tau_1)\in \mathcal{R}_{1_2}$ we have
\begin{equation*}
\begin{split}
|3\xi_1^2-2\xi_1+2\xi|&=|(2\xi_1^2-\xi_1) + (\xi_1^2-\xi_1+2\xi)|\\
&\ge |2\xi_1^2-\xi_1|-|\xi_1^2-\xi_1+2\xi|\\
&\ge |\xi_1||2\xi_1-1|-\tfrac12|\xi_1|^2\\
&\ge |\xi_1|^2-\frac12|\xi_1|^2=\tfrac12|\xi_1|^2.
\end{split}
\end{equation*}
Also we note that the hypothesis $\frac{7}{18}<2b-\tfrac12 \le a<b$ implies the following relations:
\begin{equation}\label{relations-a-b}
\frac49< b <\frac12\qquad \text{and}\qquad 4b-1\le 2a.
\end{equation}

Now we proceed to obtain \eqref{w1}. Combining the inequality $\langle\xi\rangle^{2s}\leq \langle \xi -\xi_1\rangle^{2s}\langle \xi_1\rangle^{2|s|}$  with  Lemma \ref{lemagtv1}  we get
\begin{equation}\label{first-estimation-bound-bil-uv}
w_1(\tau,\xi)\lesssim \frac{1}{\langle\tau+\xi^2\rangle^{2a}}\int \frac{\chi_{\mathcal{R}_1}\langle\xi_1\rangle^{2|s|-2k}}{\langle \tau+\xi^2-\xi_1^3+\xi_1^2-2\xi\xi_1\rangle^{4b-1}}
d\xi_1,\end{equation}
since $b>1/4$ and $a>0$. Next we estimate \eqref{first-estimation-bound-bil-uv} in each one of the sub-regions $\mathcal{R}_{1_1}$, $\mathcal{R}_{1_2}$ and $\mathcal{R}_{1_3}$.

\noindent{\textbf{Region} $\boldsymbol{\mathcal{R}_{1_1}}$}. The boundedness of the right-hand side of \eqref{first-estimation-bound-bil-uv} follows easily taking into account that the region of integrations is bounded. 

\noindent{\textbf{Region} $\boldsymbol{\mathcal{R}_{1_3}}$}. We use $|\xi_1|>2$ and $|\xi_1|^3\lesssim |\tau+\xi^2|$ to obtain from \eqref{first-estimation-bound-bil-uv} the estimate
\begin{equation*}
w_1(\tau,\xi)\lesssim \int \frac{\langle\xi_1\rangle^{2|s|-2k}}{\langle\xi_1\rangle^{6a}\langle \tau+\xi^2-\xi_1^3+\xi_1^2-2\xi\xi_1\rangle^{4b-1}}d\xi_1.
\end{equation*}
Thus, we obtain
\begin{equation*}
w_1(\tau,\xi)\lesssim \int \frac{d\xi_1}{\langle \tau+\xi^2-\xi_1^3+\xi^2-2\xi\xi_1\rangle^{4b-1}},
\end{equation*}
since  $|s|-k\leq 3a$. Hence,  Lemma \ref{lema2} allows control the last integral, once that  $b>\frac{1}{3}$. 

\noindent{\textbf{Region} $\boldsymbol{\mathcal{R}_{1_2}}$}. Is the most complicated case. The triangular inequality yields that $$\langle \tau+\xi^2\rangle\langle \tau+\xi^2-\xi_1^3+\xi_1^2-2\xi\xi_1\rangle\geq \langle \xi_1^3-\xi_1^2+2\xi\xi_1\rangle$$
and consequently
\begin{equation}\label{estimation-bound-bil-uv-A}
\begin{split}
w_1(\xi,\tau)&\lesssim \langle\tau+\xi^2\rangle^{4b-1-2a}\int \frac{\chi_{\mathcal{R}_{1_2}}\langle\xi_1\rangle^{2|s|-2k}}{\langle \xi_1^3-\xi_1^2+2\xi\xi_1\rangle^{4b-1}}d\xi_1 \le \int \frac{\chi_{\mathcal{R}_{1_2}}\langle\xi_1\rangle^{2|s|-2k}}{\langle \xi_1^3-\xi_1^2+2\xi\xi_1\rangle^{4b-1}}d\xi_1,
\end{split}
\end{equation}
where we have used \eqref{relations-a-b}. Now, making the change of variables
$$\eta=\xi_1^3-\xi_1^2+2\xi\xi_1,\quad d\eta=(3\xi_1^2-2\xi_1+2\xi)d\xi_1$$
and using the description of $\mathcal{R}_{1_2}$  combined with \eqref{desigualdade-R12} we obtain from \eqref{estimation-bound-bil-uv-A} the following inequalities:
\begin{equation*}\begin{split}
	w_1(\xi,\tau)&\lesssim \int \frac{\chi_{\{|\eta|\leq |\xi_1|^3/2\}}\langle\xi_1\rangle^{2|s|-2k}}{|3\xi_1^2-2\xi_1+2\xi|\langle\eta\rangle^{4b-1}}d\eta \lesssim \int \frac{\chi_{\{|\eta|\leq |\xi_1|^3/2\}}\langle\xi_1\rangle^{2|s|-2k}}{|\xi_1|^2\langle\eta\rangle^{4b-1}}d\eta \lesssim \int\frac{d\eta}{\langle \eta\rangle^{\frac{2k-2|s|+2}{3}+4b-1}}.
\end{split}\end{equation*}
The last integral converges if $k-|s|\geq 2-6b$. This complete the estimate in $\mathcal{R}_1$.

\noindent{\textbf{Region} $\boldsymbol{\mathcal{R}_2}$}. In this case we have
\begin{equation}\label{crb1}
\frac{1}{2}|\xi_1^3|\leq 3|\tau_1-\xi_1^3|\lesssim \langle \tau_1-\xi_1^3 \rangle.
\end{equation}
The change of variables $\eta=\tau_1-\xi_1^2+2\xi\xi_1$ with $d\eta = 2\xi_1d\xi$ in $\mathcal{R}_2$ implies
\begin{equation}\label{crb2}
|\eta|\leq |\tau_1-\xi_1^3|+|\xi_1^3-\xi_1^2+2\xi\xi_1|\lesssim \langle \tau_1-\xi_1^3\rangle.
\end{equation}
Since $a+b>1/2$ we apply Lemma \ref{lemagtv1}, combined with (\ref{crb2}) and the relation  $|\xi_1|>2 \Longrightarrow |\xi_1|\sim \langle \xi_1 \rangle$, to obtain the following estimates: 
\begin{equation*}\begin{split}
w_2(\xi_1,\tau_1) &\lesssim \frac{\langle\xi_1\rangle^{2|s|-2k}}{\langle\tau_1-\xi_1^3\rangle^{2b}}
\int\frac{\chi_{\mathcal{R}_2}d\xi}{\langle \tau_1-\xi_1^2+2\xi\xi_1\rangle^{2a+2b-1}}\lesssim \frac{|\xi_1|^{2|s|-2k}}{\langle\tau_1-\xi_1^3\rangle^{2b}}\int_{|\eta|\lesssim | \tau_1-\xi_1^3|}\frac{(1+|\eta|)^{1-2a-2b}}{2|\xi_1|} d\eta\\
&\lesssim \frac{|\xi_1|^{2|s|-2k-1}}{\langle\tau_1-\xi_1^3\rangle^{2b}}\frac{1}{\langle\tau_1-\xi_1^3\rangle^{2a+2b-2}}=\frac{|\xi_1|^{2|s|-2k-1}}{\langle\tau_1-\xi_1^3\rangle^{2a+4b-2}}.
\end{split}
\end{equation*}
Therefore, using \eqref{crb1} and that  $\frac{1}{3}< a <b$ implies $2a+4b-2 > 6a-2 >0$, we get
\begin{equation*}\begin{split}
w_2(\xi_1,\tau_1) &\lesssim \frac{|\xi_1|^{2|s|-2k -1}}{\langle\tau_1-\xi_1^3\rangle^{6a-2}}\lesssim \frac{|\xi_1|^{2|s|-2k -1}}{|\xi_1^3|^{6a-2}}=\frac{1}{|\xi_1|^{2k-2|s|+18a-5}},
\end{split}\end{equation*}
which is bounded once that  $k-|s|>-9a+\frac{5}{2}$. 

\noindent{\textbf{Region} $\boldsymbol{\mathcal{R}_3}$}. The estimate of \eqref{w3} can be obtained in the same way of the estimate of \eqref{w2}.

Then, the proof of Proposition \ref{acoplamento1} is completed. 

\subsection{Proof of Proposition \ref{estimativaw}}
In view of  Proposition \ref{acoplamento1} it suffices to prove Proposition \ref{estimativaw} under the assumption $|\tau|>10|\xi|^2$, which implies that 
$\langle\tau+\xi^2\rangle\sim \langle\tau\rangle$.

Thus, arguing as in the proof of Proposition \ref{acoplamento1} we need to show that
\begin{equation*}
	w(\xi,\tau):=\frac{\chi_{\{|\tau|>10|\xi|^2\}}}{\langle\tau+\xi^2\rangle^{2a-s}}\int\int\frac{d\tau_1d\xi_1}{\langle \tau_1-\xi_1^3\rangle^{2b}\langle \xi_1\rangle^{2k}\langle\tau-\tau_1+(\xi-\xi_1)^2\rangle^{2b}\langle \xi-\xi_1\rangle^{2s}}
\end{equation*}
is bounded. In order to estimate $w(\xi,\tau)$ we consider two cases.

\noindent{\textbf{Case 1:} $\boldsymbol{k\geq 0}$}. Using that $s, k$ are nonnegative and Lemmas \ref{lemanovo-lema2} and \ref{lemagtv1} we see that
\begin{eqnarray*}
	w(\xi,\tau)\lesssim\langle \tau\rangle^{s-2a}\int\frac{d\xi_1}{\langle \tau+\xi^2-\xi_1^3+\xi_1^2-2\xi\xi_1\rangle^{4b-1}}\leq c,
\end{eqnarray*}
where we have used that $s\leq 2a$ and $b>\frac{1}{3}$.

\noindent{\textbf{Case 2:} $\boldsymbol{k< 0}$}.
We treat separately the cases $|\xi_1|\leq 2|\xi-\xi_1|$ and $2|\xi-\xi_1|\leq |\xi_1|$. If $|\xi_1|\leq 2|\xi-\xi_1|$, then
\begin{equation*}
w(\xi,\tau)\lesssim\int \frac{d\xi_1}{\langle\tau+\xi^2-\xi_1^3+\xi^2-2\xi\xi_1\rangle^{4b-1}}\leq c_b,
\end{equation*}
where we have used that $s\leq2a$, $k+s\geq 0$ and $b>\frac{1}{3}$.

On the other hand, in the case $2|\xi-\xi_1|\leq |\xi_1|$ we have that $|\tau|> 10|\xi|^2\geq \frac{5}{2}|\xi_1|^2$ and hence it follows that $\langle\xi_1\rangle^{-2k}\lesssim \langle \tau\rangle^{-k}$. Thus
\begin{equation*}
w(\xi,\tau)\lesssim\langle\tau\rangle^{s-k-2a}\int\frac{d\xi_1}{\langle\tau+\xi^2-\xi_1^3+\xi^2-2\xi\xi_1\rangle^{4b-1}}\leq c,
\end{equation*}
provided $k-s\geq -2a$. This completes the proof of Proposition \ref{estimativaw}.
\subsection{Proof of Proposition \ref{acoplamento2}}

Let $\tau=\tau_1-\tau_2$, $\xi=\xi_1-\xi_2$, $\sigma=\tau-\xi^3$,  $\sigma_1=\tau_1+\xi_1^2$, $\sigma_2=\tau_2+\xi_2^2$. 
Following the argument of \cite{CL} and \cite{YW} we need to show that 
\begin{equation}\label{moro1}
\int\int\int\int_{\mathbb{R}^4}\frac{|\xi|\langle\xi\rangle^k f(\xi_1,\tau_1)g(\xi_2,\tau_2)\overline{\phi}(\xi,\tau)}{\langle\sigma\rangle^a\langle\sigma_1\rangle^b\langle\xi_1\rangle^s\langle\sigma_2\rangle^b\langle\xi_2\rangle^s}d\tau_1 d \xi_1 d\tau d\xi\leq \|\phi\|_{L^2}\|f\|_{L^2}\|g\|_{L^2}.
\end{equation}
Dividing the domain of integration we write the left hand side of \eqref{moro1} as 
\begin{equation}
Z_1+Z_2+Z_3+Z_4:=\int\int\int\int_{\mathcal{R}_1}+\int\int\int\int_{\mathcal{R}_2}+\int\int\int\int_{\mathcal{R}_3}+\int\int\int\int_{\mathcal{R}_4},
\end{equation}
where
\begin{eqnarray*}
	& &\mathcal{R}_1=\{(\xi,\xi_1,\tau,\tau_1)\in \mathbb{R}^4; |\xi|\leq 10  \},\\
	& & \mathcal{R}_2=\{(\xi,\xi_1,\tau,\tau_1)\in \mathbb{R}^4; |\xi|> 10,\ |\xi_1|>2|\xi_2|  \} ,\\
	& &\mathcal{R}_3=\{(\xi,\xi_1,\tau,\tau_1)\in \mathbb{R}^4; |\xi|> 10,\ |\xi_2|>2|\xi_1|  \},\\
	& & \mathcal{R}_4=\{(\xi,\xi_1,\tau,\tau_1)\in \mathbb{R}^4; |\xi|> 10,\ \frac{|\xi_2|}{2}\leq|\xi_1|\leq 2|\xi_2|  \} .
\end{eqnarray*}
\noindent{\textbf{Region} $\boldsymbol{\mathcal{R}_{1}}$}. We integrate first in $\xi$ and $\tau$ and we use the Cauchy-Schwarz inequality to obtain
\begin{equation*}
Z_1^2\lesssim\|g\|_{L_{\tau_2}^2L_{\xi_2}^2}^2\|f\|_{L_{\tau_1}^2L_{\xi_1}^2}^2\|\phi\|_{L_{\tau}^2L_{\xi}^2}^2\left\| \frac{1}{\langle \xi_1\rangle^{2s}\langle\sigma_1\rangle^{2b}}  \int\int\frac{\chi_{\{|\xi|\leq 10\}}d\tau d\xi}{\langle \sigma\rangle^{2a}\langle \sigma_2\rangle^{2b}\langle \xi_2\rangle^{2s}}\right\|_{L_{\xi_1}^{\infty}L_{\tau_1}^{\infty}}.
\end{equation*}
Using the facts $s\geq 0,\, a\leq b$ and Lemma \ref{lemagtv1} we obtain
\begin{equation*}
\frac{1}{\langle \xi_1\rangle^{2s}\langle\sigma_1\rangle^{2b}}  \int\int\frac{\chi_{\{|\xi|\leq 10\}}d\tau d\xi}{\langle \sigma\rangle^{2a}\langle \sigma_2\rangle^{2b}\langle \xi_2\rangle^{2s}}\lesssim \int\frac{\chi_{\{|\xi|\leq 10\}}d\xi}{\langle\xi^3-\xi^2+2\xi\xi_1-\tau_1-\xi_1^2\rangle^{4a-1}},
\end{equation*}
that is bounded by Lemma \ref{lema2} provided $\frac{1}{4}< a<\frac{1}{2}.$

For the estimates of $Z_2,\ Z_3$ and $Z_4$ we will use the following algebraic relation
\begin{equation}\label{Wualgebrico}
(\tau-\xi^3)-(\tau_1+\xi_1^2)+(\tau_2+\xi_2^2)=-\xi^3-\xi_1^2+(\xi_1-\xi)^2=-\xi(\xi^2-\xi+2\xi_1).
\end{equation}

\noindent{\textbf{Region} $\boldsymbol{\mathcal{R}_{2}}$}. In this case we have $\frac{1}{2}|\xi_1|\leq|\xi|\leq \frac{3}{2}|\xi_1|$, then 
\begin{equation}\label{021211}
|\xi^2-\xi+2\xi_1|\geq |\xi|^2-|\xi-2\xi_1|\geq |\xi|^2-5|\xi|\geq \frac{1}{2}|\xi|^2,
\end{equation}
since $|\xi|>10$.

Combining \eqref{Wualgebrico} and \eqref{021211} we see that
\begin{equation}\label{20111}
|\xi|^3\lesssim \max\{|\sigma|,|\sigma_1|, |\sigma_2|\} .
\end{equation}
We subdivide $\mathcal{R}_2$ in $\mathcal{R}_2=\mathcal{R}_{2_1}\cup\mathcal{R}_{2_2}\cup\mathcal{R}_{2_3}$, where
\begin{eqnarray*}
	& &\mathcal{R}_{2_1}=\{(\xi,\xi_1,\tau,\tau_1)\in \mathcal{R}_2;\max\{|\sigma|,|\sigma|, |\sigma_2|\}=|\sigma|\},\\
	& &\mathcal{R}_{2_2}=\{(\xi,\xi_1,\tau,\tau_1)\in \mathcal{R}_2;\max\{|\sigma|,|\sigma_1|, |\sigma_2|\}=|\sigma_1|\},\\
	& &\mathcal{R}_{2_3}=\{(\xi,\xi_1,\tau,\tau_1)\in \mathcal{R}_2;\max\{|\sigma|,|\sigma_2|, |\sigma_2|\}=|\sigma_2|\}.
\end{eqnarray*}
To estimate in $\mathcal{R}_{2_1}$, we integrate first in $\tau_2$ and $\xi_2$, we use Cauchy-Schwarz inequality and we use $|\xi_1|\sim |\xi|$ to control the integral by
\begin{equation}\label{paralamas2}
\chi_{\{|\xi|>10\}}\frac{\langle\xi\rangle^{2k-2s+2}}{\langle\sigma\rangle^{2a}}\int\int\frac{d\tau_2d\xi_2}{\langle\sigma_1\rangle^{2b}\langle\sigma_2\rangle^{2b}}.
\end{equation}
Now using Lemma \ref{lemagtv1} we control \eqref{paralamas2} by
\begin{equation}\label{paralamas3}
\chi_{\{|\xi|>10\}}\frac{\langle\xi\rangle^{2k-2s+2}}{\langle\sigma\rangle^{2a}}\int\frac{d\xi_2}{\langle\tau+\xi^2+2\xi\xi_2\rangle^{4b-1}}.
\end{equation}
Changing of variables $\eta=\tau+\xi^2+2\xi\xi_2$, we have $d\eta=2\xi d\xi_2$ and $|\eta|\leq 2|\sigma|$ in $\mathcal{R}_{2_1}$. Then by \eqref{20111} we control \eqref{paralamas3} by
\begin{equation}\label{paralamas4}
\chi_{\{|\xi|>10\}}\frac{\langle\xi\rangle^{2k-2s+1}}{\langle\sigma\rangle^{2a}}\int_{|\eta|\leq 2|\sigma|}\frac{d\eta}{\langle\eta\rangle^{4b-1}}\lesssim\frac{\langle\xi\rangle^{2k-2s+1}}{ \langle\sigma\rangle^{4b-2+2a}}\lesssim \langle\xi\rangle^{2k-2s-12b-6a+7},
\end{equation}
that is bounded provided $k-s\leq 6b+3a-\frac{7}{2}$.

In region $\mathcal{R}_{22}$ we use $|\xi|\sim |\xi_1|$, Lemmas \ref{lemagtv1} and \ref{lema2} to obtain
\begin{eqnarray*}
& &\frac{1}{\langle\xi_1\rangle^{2s}\langle \sigma_1\rangle^{2b}}\int\int\frac{\chi_{\{|\xi^3|\leq |\sigma_1|\}}|\xi|^2\langle\xi\rangle^{2k}d\tau d\xi}{\langle\sigma\rangle^{2a}\langle\sigma_2\rangle^{2b}}\lesssim\int\int\frac{\langle\xi\rangle^{2k-2s+2-6b}d\tau d\xi}{\langle\sigma\rangle^{2a}\langle\sigma_2\rangle^{2b}}\\
& &\quad\lesssim\int\int\frac{d\tau d\xi}{\langle\xi_2\rangle^{2s}\langle\sigma\rangle^{2a}\langle\sigma_2\rangle^{2b}}\lesssim \int\frac{d\xi}{\langle \xi^3-\xi^2+2\xi\xi_1-\tau_1-\xi_1^2\rangle^{4a-1}}\lesssim 1,
\end{eqnarray*}
where we have used that  $k-s\leq 3b-1$ and $\frac{1}{3}<a<\frac{1}{2}$.

The estimate in the region $\mathcal{R}_{23}$ can be obtained in the same way of the region $\mathcal{R}_{22}$.

\noindent{\textbf{Region} $\boldsymbol{\mathcal{R}_{3}}$}. The estimate of $Z_3$ follows the same ideas of $Z_2$. In effect, in $\mathcal{R}_3$  we have $\frac{|\xi_2|}{2}\leq|\xi|\leq 2|\xi_2|$. Then 
\begin{equation*}
|\xi^2-\xi+2\xi_1|\geq |\xi|^2-|\xi+2\xi_1|\geq |\xi|^2-3|\xi|\geq \frac{|\xi|^2}{2},
\end{equation*}
since $|\xi|>10$. Therefore the estimate can be done by separating in three cases, like in $Z_2$.

\noindent{\textbf{Region} $\boldsymbol{\mathcal{R}_{4}}$}. Finally, we estimate $Z_4$. We subdivide $\mathcal{R}_4=\mathcal{R}_{4_1}\cup\mathcal{R}_{4_2}$, where
$$\mathcal{R}_{4_1}=\{(\xi,\xi_1,\tau,\tau_1)\in \mathcal{R}_4;|\xi_1|\geq 2|\xi^2-\xi+2\xi_1|\}\: \ \text{and}\: \  \mathcal{R}_{4_2}=\mathcal{R}_4\setminus \mathcal{R}_{4_1}.$$
In $\mathcal{R}_{4_1}$ we have $|\xi^2|\leq |\xi^2-\xi+2\xi_1|+|2\xi_1-\xi|\leq 6|\xi_1|$. Then we need to show that the function
\begin{equation*}
	w_{41}(\xi_1,\tau_1):=\frac{1}{\langle\xi_1\rangle^{2s}\langle\sigma_1\rangle^{2b}}\int\int\frac{\chi_{\{|\xi|^2\leq 6|\xi_1|\}}|\xi|^2\langle\xi\rangle^{2k}d\tau d\xi}{\langle\xi_2\rangle^{2s}\langle\sigma\rangle^{2a}\langle\sigma_2\rangle^{2b}}
\end{equation*}
is bounded.
Using $a<b$, $|\xi_1|\sim |\xi_2|$ and Lemma \ref{lemagtv1} we get
\begin{equation}
	w_{41}(\xi_1,\tau_1)\lesssim \frac{1}{\langle\sigma_1\rangle^{2b}}\int\frac{\chi_{\{\xi|^2\leq 6|\xi_1|\}}\langle\xi\rangle^{2k+2-8s} d\xi}{\langle -\xi^3+\tau_1+(\xi-\xi)^2\rangle^{4a-1}}.
\end{equation}

Using  $4a-1<2b$ and triangular inequality we obtain
\begin{equation*}
	\langle\sigma_1\rangle^{4a-1}\langle-\xi^3+\tau_1+(\xi-\xi)^2 \rangle^{4a-1}\geq\langle -\xi^3-2\xi\xi_1+\xi^2\rangle^{4a-1}.
\end{equation*}
It follows that
\begin{equation}\label{05121}
	w_{41}(\xi_1,\tau_1)\leq \int\frac{\chi_{\{|\xi|^2\leq c |\xi_1|\}}\langle \xi\rangle^{2k-8s+2}d\xi}{\langle \xi^3-\xi^2+2\xi_1\xi\rangle^{4a-1}}.
\end{equation}
Now, assume $|\xi^2-\xi+2\xi_1|>1$. Then $\langle\xi^3-\xi^2+2\xi\xi_1\rangle \sim \langle \xi^2-\xi+2\xi_1\rangle\langle\xi\rangle$, provided $|\xi|>10$. Thus \eqref{05121} is controlled by
\begin{equation*}
	c\int\frac{\chi_{\{|\xi|^2\leq |\xi_1|\}}\langle \xi\rangle^{2k-8s+3-4a}d\xi}{\langle \xi^2-\xi+2\xi_1\rangle^{4a-1}}\lesssim\langle \xi_1\rangle^{k-4s+3/2-2a}\int\frac{d\xi}{\langle \xi^2-\xi+2\xi_1\rangle^{4a-1}}\lesssim 1,
\end{equation*} 
where we have used that $k\leq 4s+2a-3/2$ and $a>\frac{3}{8}$.

If $|\xi^2-\xi+2\xi_1|\leq 1$, we controlled \eqref{05121} making the change of variables $\eta=\xi^3-\xi^2+2\xi\xi_1.$ Then  $|\eta|\leq |\xi|\leq c |\xi_1|^{1/2}$ and 
\begin{equation*}
	\left|\frac{d\eta}{d \xi}\right|=|3\xi^2-2\xi+2\xi_1|\geq |2\xi^2-\xi|-|\xi^2-\xi+2\xi_1|\geq 2|\xi|^2-|\xi|-1\geq 2|\xi|^2-2|\xi|\geq \frac{1}{2}|\xi|^2,
\end{equation*}
provided $|\xi|>10$.

Therefore, \eqref{05121} is bounded by
\begin{equation*}
	\int_{|\eta|<c|\xi_1|^{1/2}}\frac{\langle\xi_1\rangle^{k-4s}d\eta}{\langle\eta\rangle^{4a-1}}\lesssim\langle \xi_1\rangle^{k-4s+1-2a}\lesssim 1,
\end{equation*}
for $k\leq 4s+2a-1$.

Now we estimate $Z_4$ in the set $\mathcal{R}_{4_2}$. In this case we use the algebraic relation \eqref{Wualgebrico} to subdivide $\mathcal{R}_{42}$ in following cases:
\begin{equation}
|\sigma|\gtrsim |\xi||\xi_1|,\ |\sigma_1|\gtrsim|\xi||\xi_1|,\ \text{or}\ |\sigma_2|\gtrsim |\xi||\xi_1|. 
\end{equation}
In each cases above we split our analysis in the situations:
$$|\xi|^2\leq 10|\xi_1| \ \text{and}\ 10|\xi_1|\leq |\xi|^2.$$
Initially we assume $$\max{\{|\sigma|,|\sigma_1|,|\sigma_2|\}}=|\sigma|\ \text{and}\ |\xi|^2\leq 10|\xi_1|.$$

In this case we have
\begin{equation}\label{05124}
\frac{\langle \xi\rangle^{2k+2-8s}}{\langle \sigma\rangle^{2a}}\int\int\frac{d\tau_2d\xi_2}{\langle \sigma_1\rangle^{2b}\langle \sigma_2\rangle^{2b}}\lesssim \frac{\langle \xi\rangle^{2k+2-8s}}{\langle \sigma\rangle^{2a}}\int\frac{d\xi_2}{\langle 2\xi\xi_2-\xi^3-\tau-2\xi^2\rangle^{4b-1}}.
\end{equation}
Now, we change the variables $\eta=2\xi\xi_2-\xi^3-\tau-2\xi^2$. Then $d\eta =2\xi d\xi_2 $ and $|\eta|\leq c |\sigma|$. Thus, the right hand side of \eqref{05124} is estimated by
\begin{equation*}
c\frac{\langle \xi\rangle^{2k+1-8s}}{\langle \sigma\rangle^{2a}}\int_{|\eta|\leq c |\sigma|}\frac{d\eta}{\langle \eta\rangle^{4b-1}}\lesssim \frac{\langle \xi\rangle^{2k+1-8s}}{\langle \sigma\rangle^{2a+4b-2}}\lesssim\frac{\langle \xi\rangle^{2k+1-8s}}{\langle \xi\rangle^{2a+4b-2}\langle\xi\rangle^{4a+8b-4}}\lesssim 1
\end{equation*} 
under the condition that $k\leq4s+3a+6b-7/2$.

The case $\max{\{|\sigma|,|\sigma_1|,|\sigma_2|\}}=|\sigma_2|$ and $|\xi|^2<10 |\xi_1|$ can be treated in the same way.

Therefore, we treat the case $10|\xi_1|\leq |\xi|^2$. Using $|\xi^2-\xi+2\xi_1|\sim |\xi|^2$ and algebraic relation \eqref{Wualgebrico} we see that
\begin{equation*}
3\max{\{|\sigma|,|\sigma_1|,|\sigma_2|\}}\geq |\xi||\xi^2-\xi+2\xi_1|\geq |\xi|^3.
\end{equation*}

Assuming $\max{\{|\sigma|,|\sigma_1|,|\sigma_2|\}}=|\sigma|$ we have 
\begin{eqnarray*}
	& &\frac{|\xi|^2\langle\xi\rangle^{2k}}{\langle\sigma\rangle^{2a}}\int\int\frac{d\tau_2d\xi_2}{\langle \xi_1\rangle^{2s}\langle \sigma_1\rangle^{2b}\langle\xi_2 \rangle^{2s}\langle \sigma_2\rangle^{2b}}\lesssim 	\frac{\langle\xi\rangle^{2k+2-4s}}{\langle\sigma\rangle^{2a}}\int\int\frac{\chi_{\{|\xi_1|\leq 10|\xi|^2\}}d\tau_2d\xi_2}{\langle \sigma_1\rangle^{2b}\langle \sigma_2\rangle^{2b}}\\
	& &\quad \lesssim \frac{\langle\xi\rangle^{2k+2-4s}}{\langle\sigma\rangle^{2a}}\int\frac{d\xi_2}{\langle2\xi\xi_2-\xi^3-\tau-2\xi^2\rangle^{4b-1}}\lesssim \frac{\langle\xi\rangle^{2k+1-4s}}{\langle\sigma\rangle^{2a}}\int_{|\eta|\leq |\sigma|}\frac{d\eta}{\langle\eta\rangle^{4b-1}}\\
	& &\quad \lesssim \frac{\langle\xi\rangle^{2k+1-4s}}{\langle\sigma\rangle^{2a+4b-2}}\lesssim \frac{\langle\xi\rangle^{2k+1-4s}}{\langle\xi\rangle^{6a+12b-6}}\lesssim 1,
\end{eqnarray*}
where we have used that $k \leq 2s+3a+6b-\frac{7}{2}$.

Now assume $\max{\{|\sigma|,|\sigma_1|,|\sigma_2|\}}=|\sigma_1|$. Then $|\sigma_1|\leq |\xi|^3$. We need to control the function
\begin{equation}
w(\xi_1,\tau_1)=\frac{1}{\langle \xi_1\rangle^{2s} \langle\tau_1+\xi_1^2\rangle^{2b}}\int\int\frac{|\xi|^2\langle\xi\rangle^{2k}d\tau d\xi}{\langle \tau-\xi^3\rangle^{2a}\langle \tau_2+\xi_2^2\rangle^{2b}\langle \xi_2\rangle^{2s}}.
\end{equation}
Using Lemmas  \ref{lemagtv1} and \ref{lemanovo-lema2} we control $w_1$ by
\begin{equation*}
c\langle \xi\rangle^{2k-4s+2-6b}\int\frac{d\xi}{\langle\xi^3-\xi^2+2\xi\xi_1+\tau_1+\xi_1^2\rangle^{4a-1}},
\end{equation*}
that is bounded provided $a>\frac{3}{8}$ and $k\leq 2s-1+3b$.

This finish the proof of Proposition \ref{acoplamento2}.

\subsection{Proof of Proposition \ref{estimativau}}

Let $\tau=\tau_1-\tau_2$, $\xi=\xi_1-\xi_2$, $\sigma=\tau-\xi^3$,  $\sigma_1=\tau_1+\xi_1^2$, $\sigma_2=\tau_2+\xi_2^2$.

\textbf{Proof of part (a).} From Proposition \ref{acoplamento2} we can assume $|\tau|>10|\xi|^3$.
Arguing as in the proof of Proposition \ref{acoplamento2} we need to show that the function
\begin{equation}
w(\tau,\xi)=\chi_{\{|\tau|>10|\xi|^3\}}\frac{|\xi|^2\langle\tau \rangle^{\frac{2k}{3}}}{\langle\sigma\rangle^{2a}}\int\int\frac{d\tau_2 d\xi_2}{\langle\xi_1\rangle^{2s}\langle\sigma_1\rangle^{2b}\langle\sigma_2\rangle^{2b}\langle\xi_2\rangle^{2s}},
\end{equation}
is bounded.

Assuming $|\xi|\leq 1$, we have that $\langle\tau-\xi^3\rangle\sim \langle\tau\rangle$. Then 
\begin{equation}\label{10123}
\chi_{\{|\tau|>10|\xi|^3\}}\frac{|\xi|^2\langle\tau \rangle^{\frac{2k}{3}}}{\langle\tau+\xi^3\rangle^{2a}}\lesssim \langle\tau \rangle^{\frac{2k}{3}-2a}\lesssim 1,
\end{equation}
for $k\leq 6a$.  

Using Lemma \ref{lemagtv1} we obtain
\begin{equation*}
\int\int\frac{d\tau_2 d\xi_2}{\langle\xi_1\rangle^{2s}\langle\sigma_1\rangle^{2b}\langle\sigma_2\rangle^{2b}\langle\xi_2\rangle^{2s}}\lesssim \int\frac{{d\xi_2}}{\langle\xi_2\rangle^{2s}\langle\xi_1\rangle^{2s}\langle\tau+\xi^2+2\xi\xi_2\rangle^{4b-1}}\lesssim\int\frac{{d\xi_2}}{\langle\xi_2\rangle^{2s}\langle\xi_1\rangle^{2s}}\lesssim 1,
\end{equation*}
where we have used $\frac{1}{4}<b<\frac{1}{2}$ and $s>\frac{1}{4}$. 

Now, if $|\xi|\geq 2$ we have
\begin{equation}\label{f1}
\frac{|\xi|^2\langle\tau \rangle^{\frac{2k}{3}}}{\langle\tau+\xi^3\rangle^{2a}}\lesssim \langle\tau \rangle^{\frac{2k}{3}-2a}|\xi|^{2}.
\end{equation} 

Since $\frac{1}{4}<b<\frac{1}{2}$ and $4b-1+2s>1$, we apply the Lemma \ref{lemagtv1} to get
\begin{eqnarray}
& &\int\int\frac{d\tau_2 d\xi_2}{\langle\xi_1\rangle^{2s}\langle\tau_1+\xi^2\rangle^{2b}\langle\tau_2+\xi_2^2\rangle^{2b}\langle\xi_2\rangle^{2s}}\lesssim\int\frac{d\xi_2}{\langle\xi_2\rangle^{2s}\langle \tau+\xi^2+2\xi\xi_2\rangle^{4b-1}}\nonumber\\
& &\quad\lesssim|\xi|^{1-4b}\int\frac{d\xi_2}{(\frac{1}{|\xi|}+|\xi_2|)^{2s}(\frac{1}{|\xi|}+|\frac{\tau+\xi^2}{2\xi}+\xi_2|)^{4b-1}}\lesssim \frac{|\xi|^{2s-1}}{\langle \tau+\xi^2\rangle^{4b+2s-2}}\lesssim \frac{|\xi|^{2s-1}}{\langle \tau\rangle^{4b+2s-2}}.\label{f2}
\end{eqnarray}

Combining (\ref{f1}) with (\ref{f2}) and using $|\tau-\xi^3|\sim |\xi|^3$ we get
\begin{equation*}
w(\tau,\xi)\lesssim\langle\tau \rangle^{\frac{2k}{3}-2a-4b-2s+2}|\xi|^{2s+1}\lesssim\langle\tau \rangle^{\frac{2k}{3}-2a-4b-2s+2+\frac{2s}{3}+\frac{1}{3}}\lesssim 1,
\end{equation*}
where we have used that $k\leq 2s+6b+3a-\frac{7}{2}$, which proves the part (a) of the Proposition \ref{estimativau}.

\textbf{Proof of part (b).} Arguing as in the proof of Proposition \ref{acoplamento2} and using $k\leq 0$ we need to show that the functions
$$w_1(\xi_1,\tau_1)=\frac{\langle\tau\rangle^{\frac{2k}{3}}}{\langle \xi_1\rangle^{2s}\langle\sigma_1\rangle^{2b}}  \int\int\frac{\chi_{|\xi|<2}d\tau d\xi}{\langle \sigma\rangle^{2a}\langle \sigma_2\rangle^{2b}\langle \xi_2\rangle^{2s}}$$
and
\begin{equation*}
w_2(\tau,\xi)=\chi_{\{|\xi|>2,\ 10|\tau|\leq |\xi|^3\}}\frac{|\xi|^2}{\langle\sigma\rangle^{2a}}\int\int\frac{d\tau_2d\xi_2}{\langle \xi_1\rangle^{2s}\langle \sigma_1\rangle^{2b}\langle\xi_2 \rangle^{2s}\langle \sigma_2\rangle^{2b}}
\end{equation*}
are bounded.

The estimate of $w_1$ can be obtained as in the estimate of $Z_1$ in the proof of Proposition \ref{acoplamento2}. 

To estimate $w_2$ we use $|\sigma|\sim |\xi|^{3}$. It follows that
\begin{equation}\label{21031}
w_2(\tau,\xi)\leq\chi_{\{|\xi|>2,\ 10|\tau|\leq |\xi|^3\}}\frac{|\xi|^2}{\langle\xi\rangle^{6a}}\int\int\frac{d\tau_2d\xi_2}{\langle \xi_1\rangle^{2s}\langle \sigma_1\rangle^{2b}\langle\xi_2 \rangle^{2s}\langle \sigma_2\rangle^{2b}}.
\end{equation}

Arguing as in the proof of part (a) we can control the right hand side of \eqref{21031} by $|\xi|^{2s+1-6a}$, that is bounded provided $s\leq 3a-\frac{1}{2}$, which proves the proposition.

\section{\textbf{Proof of the Main Results}}\label{prova-teoremas}Here we show the proof of the main results annunciated of this work. 
\subsection{Proof of the Theorem  \ref{teorema1}}\label{provadoteorema1}
Let $(s,k)\in \mathcal{D}\cup \mathcal{D}_0$. Consider $\beta=0$ for $(s,k)\in \mathcal{D}_0$ and $\beta$ any real number for $(s,k)\in \mathcal{D}$. Choice $a=a(s,k)<b=b(s,k)<\frac{1}{2}$ such that the non-linear estimates of Lemmas \ref{trilinear}, \ref{bilinear1}, and Propositions \ref{acoplamento1}, \ref{acoplamento2} and \ref{estimativau} (part (a)) are valid. Set $d=-a$.

Let $\tilde{u}_0$, $\tilde{v}_0$, $\tilde{f}$ and $\tilde{g}$ extensions of $u_0$, $v_0$, $f$ and $g$ such that $\|\tilde{u}_0\|_{H^{s}(\mathbb{R})}\leq c\|u_0\|_{H^s(\mathbb{R}^{+})}$, $\|\tilde{v}_0\|_{H^{k}(\mathbb{R})}\leq c\|v_0\|_{H^k(\mathbb{R}^+)}$, $\|\tilde{f}\|_{H^{\frac{2s+1}{4}}(\mathbb{R})}\leq c\|f\|_{H^{\frac{2s+1}{4}}(\mathbb{R}^{+})}$  and $\|\tilde{g}\|_{H^{\frac{k+1}{3}}(\mathbb{R})}\leq c\|g\|_{H^{\frac{k+1}{3}}(\mathbb{R}^{+})}$.

Using (\ref{linear}), (\ref{forcings}), \eqref{lineark}, \eqref{kdvlambda+}, (\ref{nonlinears}) and (\ref{DK}) we need to obtain a fixed point for the operator $\Lambda=(\Lambda_1,\Lambda_2)$, given by
\begin{equation*}
	\Lambda_1(u,v)=\psi(t)e^{it\partial_x^2}\tilde{u}_0(x)+\psi(t)\mathcal{S}\big(\alpha \psi_T uv+\beta\psi_T|u|^2u\big)(x,t)+\psi(t)e^{-i\frac{\lambda_1\pi}{4}}\mathcal{L}_{+}^{\lambda_1}h_1(x,t)\; \text{and}
\end{equation*}
\begin{equation*}
\Lambda_2(u,v)=\psi(t)e^{-t\partial_x^3}\tilde{v}_0(x)+\psi(t)\mathcal{K}\big(\gamma\psi_T(|u|^2)_x-\frac{1}{2}\psi_T(v^2)_x\big)(x,t)+\psi(t)e^{-i\pi\lambda}\mathcal{V}_{+}^{\lambda_2}h_2(x,t),
\end{equation*}
where 
\begin{equation*}
	h_1(t)=\left[\psi(t)\tilde{f}(t)-\psi(t)e^{it\partial_x^2}\tilde{u}_0|_{x=0}-\psi(t)\mathcal{S}\big(\alpha \psi_T uv+\beta \psi_T|u|^2u\big)(0,t)\right]\bigg|_{(0,+\infty)}
\end{equation*}
and
\begin{equation*}
	h_2(t)=\left[\psi(t)\tilde{g}(t)-\psi(t)e^{-t\partial_x^3}\tilde{v}_0|_{x=0}-\psi(t)\mathcal{K}\big(\gamma\psi_T(|u|^2)_x-\frac{1}{2}\psi_T(v^2)_x\big)(0,t)\right]\bigg|_{(0,+\infty)},
\end{equation*}
where $\lambda_1=\lambda_1(s)$ and $\lambda_2=\lambda_2(k)$ will be choosing later, so that Lemmas \ref{edbf} and \ref{cof} are to be valid.

We consider $\Lambda$ in the Banach space $Z=Z(s,k)=Z_1\times Z_2$, where
$$ Z_1=\mathcal{C}\big(\mathbb{R}_t;\,H^s(\mathbb{R}_x)\big)\cap \mathcal{C}\big(\mathbb{R}_x;\,H^{\frac{2s+1}{4}}(\mathbb{R}_t)\big)\cap X^{s,b}\, \text{and}$$
$$ Z_2=\mathcal{C}\big(\mathbb{R}_t;\,H^k(\mathbb{R}_x)\big)\cap \mathcal{C}\big(\mathbb{R}_x;\,H^{\frac{k+1}{3}}(\mathbb{R}_t)\big)\cap Y^{k,b}\cap V^{\alpha}.$$ 

Initially, we will show that the functions $\mathcal{L}_{+}^{\lambda_1}h_1$ and $\mathcal{V}_{+}^{\lambda_2}h_2$ are well defined when $u\in Z_1$ and $v\in Z_2$. For this purpose, by Lemmas \ref{edbf} and \ref{cof} it suffices to show that $h_1\in H_0^{\frac{2s+1}{4}}(\R^+)$ and $h_2\in H_0^{\frac{k+1}{3}}(\R^+)$.

Let $(s,k)\in \mathcal{D}$. By hypothesis we have that $f\in H^{\frac{2s+1}{4}}(\mathbb{\R^+})$. Lemmas \ref{sobolevh0} and \ref{grupo} (b) imply
\begin{equation}\label{aux1}
\big\|(\psi(t)e^{it\partial_x^2}\tilde{u}_0|_{x=0})|_{(0,+\infty)}\big\|_{H^{\frac{2s+1}{4}}(\mathbb{R}^+)}\leq \big\|\psi(t)e^{it\partial_x^2}\tilde{u}_0|_{x=0}\big\|_{H^{\frac{2s+1}{4}}(\mathbb{R})} \leq c\|\tilde{u}_0\|_{H^s(\mathbb{R})}.
\end{equation}

Now, Lemmas \ref{sobolevh0}, \ref{nonlinears} (b), \ref{trilinear} and \ref{T}, and Proposition \ref{acoplamento1} imply
\begin{equation}\label{aux2ap}
\begin{split}
 &\left\|\psi(t)\mathcal{S}\big(\psi_T\alpha uv+\psi_T\beta|u|^2u\big)(0,t)|_{(0,+\infty)}\right\|_{H^{\frac{2s+1}{4}}(\mathbb{R}^+)}\leq\left\|\psi(t)\mathcal{S}\big(\psi_T\alpha uv+\psi_T\beta|u|^2u\big)(0,t)\right\|_{H^{\frac{2s+1}{4}}(\mathbb{R})}\\
 &\quad\leq c\|\psi_T\big(\alpha uv+\beta|u|^2u\big)\|_{X^{s,d}} \leq cT^{\epsilon}\|\alpha uv+\beta|u|^2u\|_{X^{s,d+\epsilon}}\leq cT^{\epsilon}(\| u\|_{X^{s,b}}^3+\| u\|_{X^{s,b}}\| v\|_{Y^{k,b}}),
\end{split}
\end{equation} 
for $\epsilon$ adequately small.

If $0\leq s<\frac{1}{2}$, then $\frac{1}{4}\leq\frac{2s+1}{4}<\frac{1}{2}$, and Lemma \ref{sobolevh0} show that $H^{\frac{2s+1}{4}}(\R^+)=H_0^{\frac{2s+1}{4}}(\R^+)$. Thus \eqref{aux1} and \eqref{aux2ap} shows that $h_1\in H_0^{\frac{2s+1}{4}}(\R^+)$. 

By hypothesis $g\in H^{\frac{k+1}{3}}(\R^+)$. Lemma \ref{grupok} (b) implies
\begin{equation}\label{aux3ap}
\big\|\psi(t)e^{-t\partial_x^3}\tilde{v}_0\big|_{x=0}|_{(0,+\infty)}\big\|_{H^{\frac{k+1}{3}}(\mathbb{R}^+)}\leq\big\|\psi(t)e^{-t\partial_x^3}\tilde{v}_0\big|_{x=0}\big\|_{H^{\frac{k+1}{3}}(\mathbb{R})}\leq c\|\tilde{v}_0\|_{H^k(\mathbb{R})}.
\end{equation}
Now,  Lemmas \ref{sobolevh0}, \ref{duhamelk} (b), \ref{bilinear1} and \ref{T} and Proposition \ref{acoplamento2} imply
\begin{equation}\label{aux4ap}
\begin{split}
&\big\|\psi(t)\mathcal{K}\big[\psi_T(\gamma(|u|^2)_x-\frac{1}{2}(v^2)_x)\big](0,t)|_{(0,+\infty)}\big\|_{H^{\frac{k+1}{3}}(\mathbb{R}^+)}\\
&\quad \leq\big\|\psi(t)\mathcal{K}\big[\psi_T(\gamma(|u|^2)_x-\frac{1}{2}(v^2)_x)\big](0,t)\big\|_{H^{\frac{k+1}{3}}(\mathbb{R})}\\
&\quad\leq c\big\|\psi_T\big((|u|^2)_x-\frac{1}{2}(v^2)_x\big)\big\|_{Y^{k,d}}\leq cT^{\epsilon}\big\|(|u|^2)_x-\frac{1}{2}(v^2)_x\big\|_{Y^{k,d+\epsilon}}\\
&\quad\leq cT^{\epsilon}\big(\| u\|_{X^{s,b}}^2+\| v\|_{Y^{k,b}\cap V^{\alpha}}^2\big),
\end{split}
\end{equation} 
for $\epsilon$ adequately small. 

If $-\frac{3}{4}<k<\frac{1}{2}$, then $\frac{1}{12}<\frac{k+1}{3}<\frac{1}{2}$, and Lemma $\ref{sobolevh0}$ shows that $H^{\frac{k+1}{3}}(\R^+)=H_0^{\frac{k+1}{3}}(\R^+)$. Thus \eqref{aux3ap} and \eqref{aux4ap} shows that $h_2\in H_0^{\frac{k+1}{3}}(\R^+)$. 

If $(s,k)\in \mathcal{D}_0$ (remember that in $\mathcal{D}_0$ we assume $\beta=0$) we use the same argument to prove that $h_2\in H_0^{\frac{k+1}{3}}(\R^+)$. To show that $h_1\in H_0^{\frac{2s+1}{4}}(\R^+)$ we use a similar argument combined with Lemma \ref{alta} and the compatibility condition $u(0)=f(0)$.

The next step is to show that $\Lambda$ defines a contraction map.

Lemmas \ref{sobolevh0} (for $(s,k)\in \mathcal{D}$) or \ref{alta} (for $(s,k)\in \mathcal{D}_0$ ), \ref{grupo}, \ref{edbf}, \ref{nonlinears}, Propositions \ref{acoplamento1}, \ref{estimativaw} and Lemma \ref{T} imply that
\begin{equation}\label{cont111111}
\|\Lambda_1(u,v)\|_{Z_1}\leq  c\big(\|u_0\|_{H^s(\mathbb{R}^+)}+\|f\|_{H^\frac{2s+1}{4}(\mathbb{R}^+)}+T^{\epsilon}(\|u\|_{X^{s,b}}\|v\|_{Y^{k,b}\cap V^{\alpha}}+c_1(\beta)\|u\|_{X^{s,b}}^3)\big),
\end{equation}
where $c_1(\beta)=0$ if $\beta=0$ and $c_1(\beta)=1$, if $\beta\neq 0$ and $\epsilon<<1$.

Lemmas \ref{sobolevh0}, \ref{grupok}, \ref{cof} and \ref{duhamelk}, Propositions \ref{acoplamento2}, \ref{estimativau} and Lemma \ref{T} we have that
\begin{equation}\label{cont222222}
\|\Lambda_2(u,v)\|_{Z_2}\leq  c(\|v_0\|_{H^k(\mathbb{R}^+)}+\|g\|_{H^\frac{k+1}{3}(\mathbb{R}^+)}+T^{\epsilon}\|u\|_{X^{s,b}}^2+T^{\epsilon}\| v\|_{Y^{k,b}\cap V^{\alpha}}^2).
\end{equation}

Similarly we have
\begin{eqnarray}\label{cont333333}
	\|\Lambda(u_1,v_1)-\Lambda(u_2,v_2)\|_{Z}&\leq& cT^{\epsilon}\{ \|v_1\|_{Y^{k,b}} \|u_1-u_2\|_{X^{s,b}}+ \|u_2\|_{X^{s,b}} \|v_1-v_2\|_{Y^{k,b}}  +\nonumber\\
	& &+(\|u_1\|_{X^{s,b}}+\|u_2\|_{X^{s,b}}) \|u_1-u_2\|_{X^{s,b}}+\\
	& &+ \|v_1-v_2\|_{Y^{k,b}\cap V^{\alpha}} (\|v_1\|_{Y^{k,b}\cap V^{\alpha}}+\|v_2\|_{Y^{k,b}\cap V^{\alpha}})\nonumber\\
	& & +  (\|u_1\|_{X^{s,b}}^2+\|u_2\|_{X^{s,b}}^2)     \|u_1-u_2\|_{X^{s,b}}      \}.\nonumber
\end{eqnarray}
Set the ball of $Z$:
\begin{equation*}
	B=\{ (u,v)\in Z; \|u\|_{Z_1}\leq M_1,\ \|v\|_{Z_2}\leq M_2 \},
\end{equation*}
where $M_1=2 c \left(\|u_0\|_{H^s(\mathbb{R}^+)}+\|f\|_{H^\frac{2s+1}{4}(\mathbb{R}^+)}\right)$ e $M_2=2c\left(\|v_0\|_{H^k(\mathbb{R}^+)}+\|g\|_{H^\frac{k+1}{3}(\mathbb{R}^+)}\right)$.

Restricting $(u,v)$ on the ball $B$, we have from \eqref{cont111111}, \eqref{cont222222} and \eqref{cont333333},
\begin{equation*}
	\|\Lambda_1(u,v)\|_{Z_1}\leq  \frac{M_1}{2}+cT^{\epsilon}M_1M_2,
\end{equation*}
\begin{equation*}
	\|\Lambda_2(u,v)\|_{Z_2}\leq  \frac{M_2}{2}+cT^{\epsilon}(M_1^2+M_2^2),
\end{equation*}
\begin{equation*}
	\|\Lambda(u_1,v_1)-\Lambda(u_2,v_2)\|_{Z_1}\leq cT^{\epsilon}(M_1^2+M_1+M_2)[\|u_1-u_2\|_{Z_1}+\|v_1-v_2\|_{Z_2}],
\end{equation*}
\begin{equation*}
	\|\Lambda(u_1,v_1)-\Lambda(u_2,v_2)\|_{Z_2}\leq cT^{\epsilon}[M_1\|u_1-u_2\|_{Z_1}+M_2\|v_1-v_2\|_{Z_2}].
\end{equation*}
Then we choose $T=T(M_1,M_2)$ small enough, such that
\begin{equation*}
	\|\Lambda_1(u,v)\|_{Z_1}\leq  M_1,
\end{equation*}
\begin{equation*}
	\|\Lambda_2(u,v)\|_{Z_2}\leq  M_2
\end{equation*}
and
\begin{equation*}
\|\Lambda(u_1,v_1)-\Lambda(u_2,v_2)\|_{Z}\leq \frac12 \|(u_1,v_1)-(u_2,v_2)\|_{Z}
\end{equation*}

Thus $\Lambda$ defines a contraction in $Z\cap B$ and  we obtain a fixed point in $(u,v)$ in $B$. Therefore $$(u,v):=\big(u|_{(x,t)\in \mathbb{R}^+\times (0,T)},v|_{(x,t)\in \mathbb{R}^+\times (0,T)}\big)$$
solves the IBVP \eqref{SK} in the sense of distributions.

\subsection{Proof of the Theorem \ref{teorema11}}\label{provadoteorema11}
The proof on the region $\widetilde{\mathcal{D}}\cup\widetilde{\mathcal{D}}_0$, follows the ideas of the proof of Theorem \ref{teorema1}. We only comment some modifications of the proof. The principal differences is the position of the cutoff function $\psi_T$ in the definition of operator $\Lambda_2$ and the definition of the function $h_2$. In this situation $\Lambda_2$ and $h_2$ are given by
\begin{equation*}
\Lambda_2(u,v)=\psi(t)e^{-t\partial_x^3}\tilde{v}_0(x)+\psi(t)\mathcal{K}\big[\gamma\partial_x(|\psi_Tu|^2)-\frac{1}{2}\partial_x(v)^2\big](x,t)+\psi(t)e^{-\lambda_2\pi}\mathcal{V}_{+}^{\lambda_2}h_2(x,t),
\end{equation*}
\begin{equation*}
h_2(t)=\big(\psi(t)\tilde{g}(t)-\psi(t)e^{-t\partial_x^3}\tilde{v}_0|_{x=0}-\psi(t)\mathcal{K}[(\gamma\partial_x(|\psi_Tu|^2)-\frac{1}{2}\partial_x(v)^2)](0,t)\big)\bigg|_{(0,+\infty)}.
\end{equation*}

The  absence of the cuttof function $\psi_T$ for the nonlinear term $(v^2)_x$   it's due to the fact that in this set of regularity we need to use the modified Bourgain spaces $U^{k,b}$ and $D^{\alpha}$ that don't produce a positive power of $T$  as in the classical Bourgain spaces $Y^{k,b}$, while in the previous case we don't need the use of the modified space $U^{k,b}$.

Following the steps on the proof of Theorem \ref{teorema1} and by using Proposition \ref{estimativau} we obtain
\begin{equation}\label{cont11133}
\|\Lambda_1(u,v)\|_{Z_1}\leq  c(\|u_0\|_{H^s(\mathbb{R}^+)}+\|f\|_{H^\frac{2s+1}{4}(\mathbb{R}^+)}+T^{\epsilon}\|u\|_{X^{s,b}}\|v\|_{Y^{k,b}})+T^{\epsilon}c_1(\beta)\|u\|_{X^{s,b}},
\end{equation}
\begin{equation}\label{cont1114533}
\|\Lambda_2(u,v)\|_{Z_2}\leq  c(\|v_0\|_{H^k(\mathbb{R}^+)}+\|g\|_{H^\frac{k+1}{3}(\mathbb{R}^+)}+T^{\epsilon}\|u\|_{X^{s,b}}^2+\| v\|_{Y^{k,b}\cap V^{\alpha}}^2)
\end{equation}
and
\begin{eqnarray}\label{cont22233}
\left\|\Lambda(u_1,v_1)-\Lambda(u_2,v_2)\right\|_{Z}&\leq& c\{ T^{\epsilon}\|v_1\|_{Y^{k,b}} \|u_1-u_2\|_{X^{s,b}}+ T^{\epsilon}\|u_2\|_{X^{s,b}} \|v_1-v_2\|_{Y^{k,b}}  +\nonumber\\
& &\quad+(\|u_1\|_{X^{s,b}}+\|u_2\|_{X^{s,b}}) \|u_1-u_2\|_{X^{s,b}}\nonumber\\
& &\quad+c_1(\beta)(\|u_1\|_{X^{s,b}}^2+\|u_2\|_{X^{s,b}}^2) \|u_1-u_2\|_{X^{s,b}}\\
& &\quad+ \|v_1-v_2\|_{Y^{k,b}\cap V^{\alpha}} (\|v_1\|_{Y^{k,b}\cap V^{\alpha}}+\|v_2\|_{Y^{k,b}\cap V^{\alpha}})\}.\nonumber
\end{eqnarray}
where $c_1(\beta)=0$ if $\beta=0$ and $c_1(\beta)=1$ if $\beta\neq  0 $.

Set the following ball of $Z$:
\begin{equation*}
B=\{ (u,v)\in Z; \|u\|_{Z_1}\leq M_1,\ \|v\|_{Z_2}\leq M_2 \},
\end{equation*}
where $M_1=2 c \left(\|u_0\|_{H^s(\mathbb{R}^+)}+\|f\|_{H^\frac{2s+1}{4}(\mathbb{R}^+)}\right)$ e $M_2=2c\left(\|v_0\|_{H^k(\mathbb{R}^+)}+\|g\|_{H^\frac{k+1}{3}(\mathbb{R}^+)}\right)=2c\delta$.

Restricting $(u,v)$ on the ball $B$, by \eqref{cont11133}, \eqref{cont1114533} and \eqref{cont22233} we obtain
\begin{eqnarray*}
	& &\|\Lambda_1(u,v)\|_{Z_1}\leq  \frac{M_1}{2}+cT^{\epsilon}M_1M_2,\\
	& &\|\Lambda_2(u,v)\|_{Z_2}\leq  \frac{M_2}{2}+c(T^{\epsilon}M_1^2+M_2^2),\\
	& &\|\Lambda(u_1,v_1)-\Lambda(u_2,v_2)\|_{Z_1}\leq cT^{\epsilon}(M_1^2+M_1+M_2)(\|u_1-u_2\|_{Z_1}+\|v_1-v_2\|_{Z_2}),\\
	& &\|\Lambda(u_1,v_1)-\Lambda(u_2,v_2)\|_{Z_2}\leq c(M_1T^{\epsilon}\|u_1-u_2\|_{Z_1}+M_2\|v_1-v_2\|_{Z_2}).
\end{eqnarray*}

Then, using the smallness assumption (\ref{dadospequenosRHL}) we can choice $T=T(M_1,M_2)$ enough small, such that
\begin{equation*}
\|\Lambda_1(u,v)\|_{Z_1}\leq  M_1,
\end{equation*}
\begin{equation*}
\|\Lambda_2(u,v)\|_{Z_2}\leq  M_2,
\end{equation*}
\begin{equation*}
\|\Lambda(u_1,v_1)-\Lambda(u_2,v_2)\|_{Z\cap B}\leq \frac{1}{2}\|(u_1-u_2,v_1-v_2)\|_{Z}.
\end{equation*}

Thus,  $\Lambda$ defines a contraction map in $Z\cap B$ and we obtain a fixed point $(u,v)$ in $B$. Therefore the restriction $$(u,v):=\big(u|_{(x,t)\in \mathbb{R}^+\times (0,T)},v|_{(x,t)\in \mathbb{R}^+\times (0,T)}\big)$$
is the required solution.

\subsection{Proof of the Theorem \ref{teorema2}}\label{provadoteorema2}
Let $(s,k)\in \mathcal{E}\cup \mathcal{E}_0$. Consider $\beta=0$ for $(s,k)\in \mathcal{E}_0$ and $\beta$ any real number for $(s,k)\in \mathcal{E}$. Choice $a=a(s,k)<b=b(s,k)<\frac{1}{2}$ such that the nonlinear estimates given by Lemmas \ref{trilinear}, \ref{bilinear1} and Propositions \ref{acoplamento1}, \ref{estimativaw}, \ref{acoplamento2} and \ref{estimativau}  are valid. Set $d=-a$.

Let $\tilde{u}_0$, $\tilde{v}_0$, $\tilde{f}$, $\tilde{g}$ and $\tilde{h}$ extensions of $u_0$, $v_0$, $f$, $g$ and $h$ such that $\|\tilde{u}_0\|_{H^{s}(\mathbb{R})}\leq c\|u_0\|_{H^s(\mathbb{R}^{-})}$, $\|\tilde{v}_0\|_{H^{k}(\mathbb{R})}\leq c\|v_0\|_{H^k(\mathbb{R}^-)}$, $\|\tilde{f}\|_{H^{\frac{2s+1}{4}}(\mathbb{R})}\leq c\|f\|_{H^{\frac{2s+1}{4}}(\mathbb{R}^{+})}$, $\|\tilde{g}\|_{H^{\frac{k+1}{3}}(\mathbb{R})}\leq c\|g\|_{H^{\frac{k+1}{3}}(\mathbb{R}^{+})}$ and $\|\tilde{h}\|_{H^{\frac{k}{3}}(\mathbb{R})}\leq c\|h\|_{H^{\frac{k}{3}}(\mathbb{R}^{+})}$ .

By (\ref{forcings2}), (\ref{deltaprime}), (\ref{nonlinears}) and (\ref{DK}) we need to obtain a fixed point for the operator $\Lambda=(\Lambda_1,\Lambda_2)$, where
\begin{equation*}
	\Lambda_1(u,v)=\psi(t)e^{it\partial_x^2}\tilde{u}_0(x)+\psi(t)\mathcal{S}\big(\alpha \psi_T uv+\beta\psi_T|u|^2u\big)(x,t)+\psi(t)\mathcal{L}_{-}^{\lambda}h_1(x,t),
\end{equation*}
where 
\begin{equation*}
	h_1(t)=e^{-i\frac{\lambda\pi}{4}}\big(\psi(t)\tilde{f}(t)-\psi(t)e^{it\partial_x^2}\tilde{u}_0|_{x=0}-\psi(t)\mathcal{S}(\alpha \psi_T uv+\beta\psi_T|u|^2u)(0,t)\big)\bigg|_{(0,+\infty)}
\end{equation*}
and
\begin{equation*}
	\Lambda_2(u,v)=\psi(t)e^{-t\partial_x^3}\tilde{v}_0(x)+\psi(t)\mathcal{K}\big[\gamma\psi_T(|u|^2)_x-\frac{1}{2}\psi_T(v^2)_x\big](x,t)+\psi(t)\mathcal{V}h_2(x,t)+\psi(t)\mathcal{V}^{-1}h_3(x,t),
\end{equation*}
where $h_2$ and $h_3$ are given by
\[ \left[\begin{array}{c}
h_2  \\
h_3   \end{array} \right]=\frac{1}{3}
\left[\begin{array}{cc}
2 & -1 \\
-1 & -1 \end{array} \right]
A, \]
where
\[A=\left[\begin{array}{c}
\psi(t)\big(\tilde{g}-e^{\cdot\partial_x^3}\tilde{v}_0|_{x=0}-\mathcal{K}\big[\gamma\psi_T(|u|^2)_x-\frac{1}{2}\psi_T(v^2)_x\big](0,t)\big)\bigg|_{(0,+\infty)}\\
\mathcal{I}_{\frac{1}{3}}\psi(t)\big(\tilde{h}-\partial_xe^{\cdot\partial_x^3}\tilde{v}_0-\partial_x\mathcal{K}\big[\gamma\psi_T(|u|^2)_x-\frac{1}{2}\psi_T(v^2)_x\big](0,t)\big)\bigg|_{(0,+\infty)}\end{array} \right]. \]

Note that is this set of regularity we don't need to use the family of classes Duhamel boundary operator $\mathcal{V}_-^{\lambda}$.

We consider $\Lambda$ in the space $Z=Z(s,k)=Z_1\times Z_2$, where $$Z_1=\mathcal{C}\big(\mathbb{R}_t;\,H^s(\mathbb{R}_x)\big)\cap \mathcal{C}\big(\mathbb{R}_x;\,H^{\frac{2s+1}{4}}(\mathbb{R}_t)\big)\cap X^{s,b},$$  
\begin{equation*}
Z_2=\{w\in \mathcal{C}\big(\mathbb{R}_t;\,H^k(\mathbb{R}_x)\big)\cap \mathcal{C}\big(\mathbb{R}_x;\,H^{\frac{k+1}{3}}(\mathbb{R}_t)\big)\cap Y^{k,b}\cap V^{\alpha};\partial_x w\in \mathcal{C}\big(\mathbb{R}_x;\,H^{\frac{k}{3}}(\mathbb{R}_t)\big)  \},
\end{equation*}
with norm
\begin{eqnarray*}
	\|(u,v)\|_{Z}&=&\|u\|_{Z_1}+\|v\|_{Z_2}\\
	&:=&\|u\|_{\mathcal{C}\big(\mathbb{R}_t;\,H^s(\mathbb{R}_x)\big)}+\|u\|_{\mathcal{C}\big(\mathbb{R}_x;\,H^{\frac{2s+1}{4}}(\mathbb{R}_t)\big)}+\|u\|_{X^{s,b}}\\
	& &+\|v\|_{\mathcal{C}\big(\mathbb{R}_t;\,H^k(\mathbb{R}_x)\big)}+\|v\|_{\mathcal{C}\big(\mathbb{R}_x;\,H^{\frac{k+1}{3}}(\mathbb{R}_t)\big)}+\|v\|_{ Y^{k,b}}+\|v\|_{V^{\alpha}}+\|\partial_x v\|_{\mathcal{C}\big(\mathbb{R}_x;\,H^{\frac{k}{3}}(\mathbb{R}_t)\big)}.
\end{eqnarray*}

Arguing as in the proof of Theorem \ref{teorema1} we have that the function $\mathcal{L}_{-}^{\lambda}h_1(x,t)$ is well defined.

 Now we prove that the functions $\mathcal{V}^{-1}h_2(x,t)$ and $\mathcal{V}^{-1}h_3(x,t)$ are well defined. Arguing as in the proof of Theorem \ref{teorema1} we see that
\begin{equation}\label{09121} \psi(t)\big(\tilde{g}-e^{\cdot\partial_x^3}\tilde{v}_0|_{x=0}-\mathcal{K}\big[\gamma\psi_T(|u|^2)_x-\frac{1}{2}\psi_T(v^2)_x\big](0,t)\big)\bigg|_{(0,+\infty)}\in H_0^{\frac{k+1}{3}}(\R^+).
\end{equation}
 
 By Lemmas \ref{lemaint}, \ref{grupok}, \ref{duhamelk}, and Propositions \ref{bilinear1} and \ref{acoplamento2}  we have
 \begin{equation}\label{liga1}
 \begin{split}
 &\left\|\mathcal{I}_{\frac{1}{3}}\psi(t)\big(\tilde{h}-\partial_xe^{\cdot\partial_x^3}v_0-\partial_x\mathcal{K}\big[\gamma\psi_T(|u|^2)_x-\frac{1}{2}\psi_T(v^2)_x\big](0,t)\big)\bigg|_{(0,+\infty)}\right\|_{H^{\frac{k+1}{3}}(\mathbb{R}^+)}\\
 &\quad \leq c \left\|\psi(t)\big(\tilde{h}-\partial_xe^{\cdot\partial_x^3}v_0-\partial_x\mathcal{K}\big[\gamma\psi_T(|u|^2)_x-\frac{1}{2}\psi_T(v^2)_x\big](0,t)\big)\right\|_{H^{\frac{k}{3}}(\mathbb{R})}\\
 & \quad\leq c \|h\|_{H^{\frac{k}{3}}(\mathbb{R}^+)}+\|v_0\|_{H^{k}(\mathbb{R}^-)}+\|\psi_T(v^2)_x\|_{Y^{k,-a}}+\|\psi_T(|u|^2)_x\|_{Y^{k,-a}}\\ 
 & \quad\leq c \|h\|_{H^{\frac{k}{3}}(\mathbb{R}^+)}+\|v_0\|_{H^{k}(\mathbb{R}^-)}+T^{\epsilon}\|(v^2)_x\|_{Y^{k,-a+\epsilon}}+T^{\epsilon}\|(|u|^2)_x\|_{Y^{k,-a+\epsilon}} \\
 & \quad\leq c \|h\|_{H^{\frac{k}{3}}(\mathbb{R}^+)}+\|v_0\|_{H^{k}(\mathbb{R}^-)}+T^{\epsilon}\|v\|_{Y^{k,b}}^2+T^{\epsilon}\|u\|_{X^{s,b}}^2.
 \end{split}
 \end{equation}
It follows that $$\mathcal{I}_{\frac{1}{3}}\psi(t)\big(\tilde{h}-\partial_xe^{\cdot\partial_x^3}v_0-\partial_x\mathcal{K}\big[\gamma\psi_T(|u|^2)_x-\frac{1}{2}\psi_T(v^2)_x\big](0,t)\big)\bigg|_{(0,+\infty)}\in H^{\frac{k+1}{3}}(\R^+).$$ Since $(s,k)\in \mathcal{E}\cup\mathcal{E}_0$ we have $0\leq k<\frac{1}{2}$, then $0\leq\frac{k+1}{3}<\frac{1}{2}$. Thus Lemma \ref{sobolevh0} implies that
\begin{equation}\label{09122}
 \mathcal{I}_{\frac{1}{3}}\psi(t)\big(\tilde{h}-\partial_xe^{\cdot\partial_x^3}v_0-\partial_x\mathcal{K}\big[\gamma\psi_T(|u|^2)_x-\frac{1}{2}\psi_T(v^2)_x\big](0,t)\big)\bigg|_{(0,+\infty)}\in H_0^{\frac{k+1}{3}}(\R^+). 
\end{equation}
Thus \eqref{09121} and \eqref{09122} show that the functions $\mathcal{V}^{-1}h_2(x,t)$ and $\mathcal{V}^{-1}h_3(x,t)$ are well defined. 

Now we show that $\Lambda$ defines a contraction map in a ball of $Z$.

Arguing as in the proof of Theorem \ref{teorema1} we obtain
\begin{equation*}
\|\Lambda_1(u,v)\|_{Z_1} \leq cT^{\epsilon}\| u\|_{X^{s,b}}\| v\|_{Y^{k,b}}+cT^{\epsilon}\|u\|_{X^{s,b}}^3+c\|u_0\|_{H^s(\mathbb{R}^-)}+ c\|f\|_{H_0^{\frac{2s+1}{4}}(\mathbb{R}^+)},
\end{equation*}
\begin{equation}\label{liga2}
\begin{split}
&\left\|\psi(t)e^{-t\partial_x^3}\tilde{v}_0(x)+\psi(t)\mathcal{K}[\gamma\psi_T(|u|^2)_x-\frac{1}{2}\psi_T(v^2)_x](x,t)\right\|_{ \mathcal{C}\big(\mathbb{R}_t;\,H^k(\mathbb{R}_x)\big)\cap \mathcal{C}\big(\mathbb{R}_x;\,H^{\frac{k+1}{3}}(\mathbb{R}_t)\big)\cap Y^{k,b}\cap V^{\alpha}}\\
&\quad\leq c\|v_0\|_{H^k(\mathbb{R}^-)}+T^{\epsilon}\|u\|_{X^{s,b}}^2+T^{\epsilon}\|v\|_{Y^{k,b}\cap V^{\alpha}}^2.
\end{split}
\end{equation}
Using Lemmas \ref{grupok} and \ref{duhamelk} we see that
\begin{equation}\label{liga3}
\begin{split}
&\left\|\partial_x(\psi(t)e^{-t\partial_x^3}\tilde{v}_0(x)+\psi(t)\mathcal{K}\big[\gamma\psi_T(|u|^2)_x-\frac{1}{2}\psi_T(v^2)_x\big](x,t))\right\|_{ \mathcal{C}\big(\mathbb{R}_x;\,H^{\frac{k}{3}}(\mathbb{R}_t)\big)}\\
&\quad \leq c\|v_0\|_{H^k(\mathbb{R}^-)}+T^{\epsilon}\|u\|_{X^{s,b}}^2+T^{\epsilon}\|v\|_{Y^{k,b}\cap V^{\alpha}}^2.
\end{split}
\end{equation}
Set $W_2= \mathcal{C}\big(\mathbb{R}_t;\,H^k(\mathbb{R}_x)\big)\cap \mathcal{C}\big(\mathbb{R}_x;\,H^{\frac{k+1}{3}}(\mathbb{R}_t)\big)\cap Y^{k,b}\cap V^{\alpha}.$ Arguing as in the proof of Theorem \ref{teorema1} we obtain
\begin{equation}\label{liga4}
\begin{split}
&\left\|\mathcal{V}\left(\psi(t)\big(\tilde{g}-e^{\cdot\partial_x^3}\tilde{v}_0|_{x=0}-\mathcal{K}[\gamma\psi_T(|u|^2)_x-\frac{1}{2}\psi_T(v^2)_x](0,t)\big)\bigg|_{(0,+\infty)}\right)\right\|_{W_2}\\
&\quad \leq c \|g\|_{H^{\frac{k+1}{3}}(\mathbb{R}^+)}+\|v_0\|_{H^{k}(\mathbb{R}^-)}+T^{\epsilon}\|v\|_{H^{k}(\mathbb{R})}^2+T^{\epsilon}\|u\|_{H^s(\mathbb{R})}^2.
\end{split}
\end{equation}Using the estimate of derivatives in Lemma \ref{cof}  we obtain
\begin{equation}\label{liga5}
\begin{split}
&\left\|\partial_x\mathcal{V}\left(\psi(t)\big(\tilde{g}-e^{\cdot\partial_x^3}\tilde{v}_0|_{x=0}-\mathcal{K}[\gamma\psi_T(|u|^2)_x-\frac{1}{2}\psi_T(v^2)_x](0,t)\big)\bigg|_{(0,+\infty)}\right)\right\|_{\mathcal{C}\big(\mathbb{R}_x;\,H^{\frac{k}{3}}(\mathbb{R}_t)\big)}\\
&\quad  \leq c \|g\|_{H^{\frac{k+1}{3}}(\mathbb{R}^+)}+\|v_0\|_{H^{k}(\mathbb{R}^-)}+T^{\epsilon}\|v\|_{H^{k}(\mathbb{R})}^2+T^{\epsilon}\|u\|_{H^s(\mathbb{R})}^2.
\end{split}
\end{equation}
By Lemma \ref{cof} and estimate \eqref{liga1} we have
\begin{equation}\label{liga6}
\begin{split}
& \left\|\mathcal{V}\left(\mathcal{I}_{\frac{1}{3}}\psi(t)\big(\tilde{h}-\partial_xe^{\cdot\partial_x^3}v_0-\partial_x\mathcal{K}\big[\gamma\psi_T(|u|^2)_x-\frac{1}{2}\psi_T(v^2)_x\big](0,t)\big)\bigg|_{(0,+\infty)}\right)\right\|_{W_2}\\
&\quad \leq c\|v_0\|_{H^k(\mathbb{R}^-)}+T^{\epsilon}\|u\|_{X^{s,b}}^2+T^{\epsilon}\|v\|_{Y^{k,b}\cap V^{\alpha}}^2.
\end{split}
\end{equation}
	Combining the estimates \eqref{liga2}, \eqref{liga3}, \eqref{liga4}, \eqref{liga5} and \eqref{liga6}, we obtain

	\begin{equation*}
	\|\Lambda_2(u,v)\|_{Z_2}\leq c \|v_0\|_{H^{k}(\mathbb{R}^-)}+\|g\|_{H^{\frac{k+1}{3}}(\mathbb{R}^+)}+\|h\|_{H^{\frac{k}{3}}(\mathbb{R}^+)}+T^{\epsilon}\|v\|_{Y^{k,b}}^2+T^{\epsilon}\|u\|_{X^{s,b}}^2.
	\end{equation*}

We then proceed in the manner of Theorem \ref{teorema1} to complete the proof of Theorem \ref{teorema2}.

\subsection{Proof of the Theorem \ref{teorema22}}\label{provadoteorema22}

Let $(s,k)\in \widetilde{\mathcal{E}}_1\cup \widetilde{\mathcal{E}}_{1_{0}} \cup\widetilde{\mathcal{E}}_2\cup \widetilde{\mathcal{E}}_{2_{0}}$. Choice $a=a(s,k)<b=b(s,k)<\frac{1}{2}$ such that the non-linear estimates of Lemmas \ref{trilinear}, \ref{bilinear1} and Propositions \ref{acoplamento1}, \ref{acoplamento2} and \ref{estimativau} are valid. Set $d=-a$.
Let $\tilde{u}_0$, $\tilde{v}_0$, $\tilde{f}$, $\tilde{g}$ and $\tilde{h}$ nice extensions of $u_0$, $v_0$, $f$, $g$ and $h$.
. Let $\lambda_1=\lambda_1(s)$, $\lambda_2=\lambda_2(k)$ and $\lambda_{3}=\lambda_2(k)$ such that Lemmas \ref{edbf} and \ref{cof} are valid.

By using (\ref{linear}), (\ref{forcings2}), \eqref{lineark}, \eqref{kdvlambda-}, (\ref{nonlinears}) and (\ref{DK}) we need to obtain a fixed point to operator $\Lambda=(\Lambda_1,\Lambda_2)$, given by
\begin{equation*}
\Lambda_1(u,v)=\psi(t)e^{it\partial_x^2}\tilde{u}_0(x)+\psi(t)\mathcal{S}\big(\alpha \psi_T uv+\beta\psi_T|u|^2u\big)(x,t)+\psi(t)\mathcal{L}^{\lambda_1}h_1(x,t)
\end{equation*}
and
\begin{equation*}
\Lambda_2(u,v)\!\!=\!\!\psi(t)e^{-t\partial_x^3}\tilde{v}_0(x)+\psi(t)\mathcal{K}\big[\gamma\psi_T(|u|^2)_x-\frac{1}{2}(v^2)_x\big](x,t)\!+\!\psi(t)\mathcal{V}_{-}^{\lambda_2}h_2(x,t)\!+\!\psi(t)\mathcal{V}_{-}^{\lambda_3}h_3(x,t),
\end{equation*}
where 
\begin{equation*}
h_1(t)=e^{-i\frac{\lambda\pi}{4}}\big(\tilde{f}(t)-\psi(t)e^{it\partial_x^2}\tilde{u}_0|_{x=0}-\psi(t)\mathcal{S}(\alpha \psi_T uv+\beta \psi_T|u|^2u)(0,t)\big)\big|_{(0,+\infty)},
\end{equation*}
\begin{equation*}
\left[\begin{array}{c}
h_2(t)\\
h_3(t)
\end{array}\right]=A\left[\begin{array}{c}
\psi(t)\big(\tilde{g}-e^{\cdot\partial_x^3}\tilde{v}_0|_{x=0}-\mathcal{K}[\gamma\psi_T(|u|^2)_x-\frac{1}{2}(v^2)_x](0,t)\big)\bigg|_{(0,+\infty)}\\
\mathcal{I}_{\frac{1}{3}} \psi(t)\big(\tilde{h}-\partial_xe^{\cdot\partial_x^3}\tilde{v}_0-\partial_x\mathcal{K}[\gamma\psi_T(|u|^2)_x-\frac{1}{2}(v^2)_x](0,t)\big)\bigg|_{(0,+\infty)}\end{array} \right]
\end{equation*}
and
\begin{equation*}
A=\frac{1}{2\sqrt{3}\,\text{sin}\ \pi/3(\lambda_3-\lambda_2)}\left[\begin{array}{cc}
\,\text{sin}\left(\frac{\pi}{3}\lambda_3-\frac{\pi}{6}\right)& -\,\text{sin}\left(\frac{\pi}{3}\lambda_3+\frac{\pi}{6}\right)\\
-\,\text{sin}\left(\frac{\pi}{3}\lambda_2-\frac{\pi}{6}\right)& \,\text{sin}\left(\frac{\pi}{3}\lambda_2+\frac{\pi}{6}\right)\end{array}\right].
\end{equation*}
Let $Z$ the space given in the proof of Theorem \ref{teorema2}.

Arguing as in the proof of Theorem  \ref{teorema2} we have that $\mathcal{L}h_1(x,t)$ is well defined and 
\begin{equation}\label{crbxparana1}
\|\Lambda_1(u,v)\|_{Z_1} \leq cT^{\epsilon}\| u\|_{X^{s,b}}\| v\|_{Y^{k,b}}+cc_1(\beta)T^{\epsilon}\|u\|_{X^{s,b}}^3+c\|u_0\|_{H^s(\mathbb{R}^-)}+ c\|f\|_{H_0^{\frac{2s+1}{4}}},
\end{equation}
where $c_1(\beta)=0$ if $\beta=0$ and $c_1(\beta)=1$ if $\beta\neq0$.

Arguing as in the proof of Theorem \ref{teorema2} we have that the functions $\mathcal{V}_{-}^{\lambda_2}h_2$ and $\mathcal{V}_{-}^{\lambda_3}h_3$ are well defined.

Lemma \ref{cof} implies
 \begin{equation}\label{galo1}
 \begin{split}
 & \left\|\mathcal{V}_{-}^{\lambda_2}\left(\psi(t)\big(g-e^{\cdot\partial_x^3}\tilde{v}_0|_{x=0}-\mathcal{K}\big[\gamma\psi_T(|u|^2)_x-\frac{1}{2}(v^2)_x\big](0,t)\big)\bigg|_{(0,+\infty)}\right)\right\|_{Z_2}\\
 &\quad \leq c \left\|\psi(t)\big(g-e^{\cdot\partial_x^3}\tilde{v}_0|_{x=0}-\mathcal{K}[\gamma\psi_T(|u|^2)_x-\frac{1}{2}(v^2)_x](0,t)\big)\bigg|_{(0,+\infty)}\right\|_{H_0^{\frac{k+1}{3}}(\mathbb{R}^+)},
 \end{split} 
 \end{equation}
 \begin{equation}\label{galo2}
 \begin{split}
 & \left\|\mathcal{V}_{-}^{\lambda_3}\psi(t)\mathcal{I}_{\frac{1}{3}}\left(\tilde{h}-\partial_xe^{\cdot\partial_x^3}\tilde{v}_0-\partial_x\mathcal{K}[\gamma\psi_T(|u|^2)_x-\frac{1}{2}(v^2)_x](0,t)\right)\bigg|_{(0,+\infty)}\right\|_{Z_2}\\
 &\quad \leq c \left\|\psi(t)\mathcal{I}_{\frac{1}{3}}\left(\tilde{h}-\partial_xe^{\cdot\partial_x^3}\tilde{v}_0-\partial_x\mathcal{K}[\gamma\psi_T(|u|^2)_x-\frac{1}{2}(v^2)_x](0,t)\right)\bigg|_{(0,+\infty)}\right\|_{H_0^{\frac{k+1}{3}}(\mathbb{R}^+)}
 \end{split} 
 \end{equation}
provided $-1<\lambda_2,\lambda_3<\min \{\frac{1}{2},k+\frac{1}{2}\}$.   
 
From Lemmas \ref{lemaint}, \ref{grupok}, \ref{duhamelk}, \ref{bilinear1} and Proposition \ref{acoplamento2} we have
 \begin{equation}\label{galo3}
 \begin{split}
 &\left\|\mathcal{I}_{\frac{1}{3}}\psi(t)\big(\tilde{h}-\partial_xe^{\cdot\partial_x^3}v_0-\partial_x\mathcal{K}\big[\gamma\psi_T(|u|^2)_x-\frac{1}{2}(v^2)_x\big](0,t)\big)\bigg|_{(0,+\infty)}\right\|_{H_0^{\frac{k+1}{3}}(\mathbb{R}^+)}\\
 & \quad\leq c \left\|\psi(t)\big(h-\partial_xe^{\cdot\partial_x^3}v_0-\partial_x\mathcal{K}[\gamma\psi_T(|u|^2)_x-\frac{1}{2}(v^2)_x](0,t)\big)\right\|_{H^{\frac{k}{3}}(\mathbb{R})}\\
 &\quad \leq c( \|h\|_{H^{\frac{k}{3}}(\mathbb{R}^+)}+\|v_0\|_{H^{k}(\mathbb{R}^-)}+\|(v^2)_x\|_{Y^{k,-a}}+\|(v^2)_x\|_{U^{k,-a}}+\|\psi_T(|u|^2)_x\|_{Y^{k,-a}})\\ 
 &\quad \leq c( \|h\|_{H^{\frac{k}{3}}(\mathbb{R}^+)}+\|v_0\|_{H^{k}(\mathbb{R}^-)}+\|(v^2)_x\|_{Y^{k,-a}}+\|(v^2)_x\|_{U^{s,-a}}+T^{\epsilon}\|(|u|^2)_x\|_{Y^{k,-a+\epsilon}}) \\
 & \quad\leq c (\|h\|_{H^{\frac{k}{3}}(\mathbb{R}^+)}+\|v_0\|_{H^{k}(\mathbb{R}^-)}+\|v\|_{Y^{k,b}\cap V^{\alpha}}^2+T^{\epsilon}\|u\|_{X^{s,b}}^2).
 \end{split}
 \end{equation}
  Using Lemma \ref{duhamelk} and the estimates \eqref{galo1}, \eqref{galo2}, \eqref{galo3} we obtain
 \begin{equation}\label{chavee23}
 \left\|\Lambda_2(u,v)\right\|_{Z_2}\leq c( \|v_0\|_{H^{k}(\mathbb{R}^-)}\!+\!\|g\|_{H^{\frac{k+1}{3}}(\mathbb{R}^+)}\!+\!\|h\|_{H^{\frac{k}{3}}(\mathbb{R}^+)}\!+\!\|v\|_{Y^{k,b}}^2+T^{\epsilon}\|u\|_{X^{s,b}}\|v\|_{Y^{k,b}\cap V^{\alpha}}).
 \end{equation}
 Combining \eqref{crbxparana1} with \eqref{chavee23} we obtain
 \begin{equation}\label{cont1113}
 	\begin{split}
 		\|\Lambda(u,v)\|_{Z}&\leq  c(\|u_0\|_{H^s(\mathbb{R}^-)}+\|v_0\|_{H^k(\mathbb{R}^-)}+\|f\|_{H^\frac{2s+1}{4}(\mathbb{R}^+)}+\|g\|_{H^\frac{k+1}{3}(\mathbb{R}^+)}+\|h\|_{H^\frac{k}{3}(\mathbb{R}^+)})\\
 		&  \quad +c(T^{\epsilon}\|u\|_{X^{s,b}}\|v\|_{Y^{k,b}}+\| v\|_{Y^{k,b}\cap V^{\alpha}}^2+T^{\epsilon}\| u\|_{X^{s,b}}^2).\nonumber
 	\end{split}
 \end{equation}
 We then proceed as in the proof of Theorem \ref{teorema11} to finish the proof of Theorem \ref{teorema22}.


\end{document}